\documentclass[11pt]{amsart} 
\usepackage{amssymb, amsthm, amsmath, amscd, amsfonts}
\usepackage{latexsym}
\usepackage{verbatim}
\usepackage[all]{xy}

\usepackage[english]{babel}

\usepackage{geometry}                
\usepackage{graphicx}
\usepackage{amssymb}
\usepackage{epstopdf}
\numberwithin{equation}{section}
\newtheorem{theorem}[equation]{Theorem}
\newtheorem{proposition}[equation]{Proposition}

\newtheorem{corollary}[equation]{Corollary}
\newtheorem{lemma}[equation]{Lemma}
\newtheorem{definition}[equation]{Definition}

\newtheorem{remark}[equation]{Remark}

\newcommand{\z}{\mathbb Z}       
\newcommand{\n}{\mathbb N}       

\newcommand{\ds}{\displaystyle}

\begin{document}

\title[Nearly continuous Kakutani equivalence]{The Morse Minimal System is Nearly Continuously Kakutani Equivalent to the Binary Odometer}

\author{Andrew Dykstra}
\address{Department of Mathematics\\
	Hamilton College \\
	Clinton, NY 13323 \\
	 USA}

\email{adykstra@hamilton.edu}


\author{Ay\c se \c Sah\. in}
\address{Department of Mathematics\\
	DePaul University\\
	Chicago, IL 60604\\
	USA}
	
\email{asahin@depaul.edu}



\date{\today}

\begin{abstract}

Ergodic homeomorphisms $T$ and $S$ of Polish probability spaces $X$ and $Y$ are {\em evenly Kakutani equivalent} if there is an orbit equivalence $\phi: X_0 \rightarrow Y_0$ between full measure subsets of $X$ and $Y$ such that, for some $A \subset X_0$ of positive measure, $\phi$ restricts to a measurable isomorphism of the induced systems $T_A$ and $S_{\phi(A)}$.  The study of even Kakutani equivalence dates back to the seventies, and it is well known that any two zero-entropy loosely Bernoulli systems are evenly Kakutani equivalent.  But even Kakutani equivalence is a purely measurable relation, while systems such as the Morse minimal system are both measurable and topological.  

Recently del Junco, Rudolph and Weiss studied a new relation called {\em nearly continuous Kakutani equivalence}.  A nearly continuous Kakutani equivalence is an even Kakutani equivalence where also $X_0$ and $Y_0$ are invariant $G_\delta$ sets, $A$ is within measure zero of both open and closed, and $\phi$ is a homeomorphism from $X_0$ to $Y_0$.  It is known that nearly continuous Kakutani equivalence is strictly stronger than even Kakutani equivalence, and nearly continuous Kakutani equivalence is the natural strengthening of even Kakutani equivalence to the {\em nearly continuous} category---the category where maps are continuous after sets of measure zero are removed.  In this paper we show that the Morse minimal substitution system is nearly continuously Kakutani equivalent to the binary odometer.    

\end{abstract}

\maketitle

\section{Introduction} \label{introduction} 

Even Kakutani equivalence is one of the most natural examples in the theory of restricted orbit equivalence of ergodic and finite measure preserving dynamical systems.  In this paper we study even Kakutani equivalence in the nearly continuous category.   
A nearly continuous dynamical system is given by a triple $(X,\mu,T)$, where $X$ is a Polish space, $\mu$ is a Borel probability measure on $X$, and $T:X\rightarrow X$ is an ergodic measure preserving homeomorphism.   Recall that a measurable orbit equivalence between two such systems $(X,\mu,T)$ and $(Y,\nu,S)$ is an invertible, bi-measurable, and measure preserving map $\phi:X\rightarrow Y$ that sends orbits to orbits.  A measurable orbit equivalence $\phi: X \rightarrow Y$ is a {\em nearly continuous orbit equivalence} if there exist invariant and $G_\delta$ subsets $X_0\subset X$ and $Y_0\subset Y$ of full measure so that $\phi:X_0\rightarrow Y_0$ is a homeomorphism.  

The first result in this category is the celebrated theorem of Keane and Smorodinsky \cite{KS1} that any two Bernoulli shifts of equal entropy are finitarily isomorphic, namely, that the isomorphism between them can be made a homeomorphism almost everywhere. In a later paper, Denker and Keane \cite{DK} established a general framework for studying measure preserving systems that also preserve a topological structure.   We refer the reader to a paper by del Junco, Rudolph, and Weiss \cite{JRW} for a more complete history of the area.  We only mention here that interest in the orbit equivalence theory for this category was more recently revived by the work of Hamachi and Keane in \cite{HK} where they proved that the binary and ternary odometers are nearly continuously orbit equivalent.  Their work inspired similar results for other pairs of examples (see \cite{HKR}, \cite{HKY}, \cite{R1}, \cite{R2}, \cite{RR4}, and \cite{RR3}).  These examples were later subsumed as special cases of a Dye's Theorem in this category proved by del Junco and \c Sahin \cite{JS}.   

Around the same time as a nearly continuous Dye's Theorem was established,  del Junco, Rudolph, and Weiss proved in \cite{JRW} that if one does not impose the condition that the invariant sets of full measure on which the orbit equivalence is a homeomorphism are $G_\delta$ sets, then any restricted orbit equivalence classification is exactly the same as in the measure theoretic case.  In particular, they showed that any orbit equivalence can be regularized to be a homeomorphism on a set of full measure, but could not prove that the set of full measure had any topological structure. 

The importance of the topological structure in the theory is even more striking for the study of even Kakutani equivalence.  Recall that in the measurable category two ergodic and finite measure preserving systems $(X,\mu,T)$ and $(Y,\nu,S)$ are even Kakutani equivalent if there exists a measurable orbit equivalence $\phi: X\rightarrow Y$, and measurable sets $A\subset X, B\subset Y$ with $\mu A=\nu B>0$ with the property that $\phi:A\rightarrow B$ is a measurable isomorphism of the induced transformations $T_A$ and $S_B$.  We call the orbit equivalence $\phi$ an {\it even Kakutani equivalence} between $T$ and $S$.  It follows from \cite{JRW} that any even Kakutani equivalence can be made to be a homeomorphism on a set of full measure.  In the same paper they show that if one imposes the additional condition that the sets $A$ and $B$ be {\it nearly clopen}, meaning within a set of measure zero of an open set and also of a closed set, then there is a new invariant for even Kakutani equivalence of nearly continuous dynamical systems called {\it near unique ergodicity}.  They use this new invariant to show that nearly continuous even Kakutani equivalence  is stronger than measure theoretic even Kakutani equivalence.  The example they construct is, in some sense, not natural, and begs the question whether there are any natural examples of nearly continuous systems that are measurably evenly Kakutani equivalent but not nearly continuously so.  

  Rudolph began looking for examples in the family of zero entropy Loosely Bernoulli systems.  Recall that any two zero entropy Loosely Bernoulli transformations are measurably even Kakutani equivalent.  Furthermore, many natural examples of nearly continuous systems including rotations, all adding machines,  and in fact all finite rank transformations, are Loosely Bernoulli.    In \cite{RR1}, Roychowdhury and Rudolph  proved that any two adding machines are nearly continuously even Kakutani equivalent.  Shortly after, Dykstra and Rudolph  showed in \cite{DR} that all irrational rotations are nearly continuously Kakutani equivalent to the binary odometer.  
  
  In \cite{RR1}, new machinery, called {\em templates}, was introduced to construct the nearly continuous Kakutani equivalence.  There, templates were defined using the natural topological tower structure present in adding machines.  The construction in \cite{DR} showed that the template machinery can be adapted to the case where the underlying system does not have a canonical symbolic structure.    More recently, Springer \cite{Spr} expanded on their ideas and adapted templates further to prove that all minimal isometries of compact metric spaces are nearly continuously Kakutani equivalent to the binary odometer.   Salvi \cite{Sal} adapted templates to the setting of $\mathbb R$ actions and used the machinery to prove Rudolph's Two-Step Coding Theorem in the nearly continuous category.

Each result mentioned above has required more sophisticated and technically intricate incarnations of templates.  On the other hand each proof has also established the usefulness and flexibility of the machinery.  In this paper we adapt the template machinery even further to show our main result:
\begin{theorem}
The Morse minimal system is nearly continuously even Kakutani equivalent to the binary odometer.
\end{theorem}
The version of the template machinery in this paper is designed to address the new complication of the additional tower present in the rank two  Morse system.  We believe the generalization we give here is the appropriate starting place to prove more generally that finite rank nearly continuous systems are all nearly continuously Kakutani equivalent to the binary odometer.

Finally we note that this manuscript is a culmination of work that the first author began in 2009 while he was a post-doctoral fellow working with Daniel Rudolph at Colorado State University.  The initial architecture of the constructions and the main ideas were all established collaboratively by Dykstra and Rudolph.  The second author joined the project after the untimely death of Rudolph in 2010, and the manuscript was completed in 2014.

\section{Template Machinery}

In this section, deferring some formal definitions until later, we give an overview of the construction and introduce templates.  Let $(X,T,\mu)$ denote the Morse minimal system and $(Y,S,\nu)$ the binary odometer.  Recall that each system has a canonical refining, generating sequence of clopen partitions that are given by the finite rank structure of each system.  The Morse system is rank two, so at each stage the partition is defined by a pair of towers. The odometer is rank one, so the sequence of partitions is defined by a sequence of single towers.  The construction of the orbit equivalence uses an inductive \lq\lq back and forth\rq\rq procedure.  Intuitively, at each stage we need to construct a set map from the levels of the tower of one system to the levels of the tower of the other system, switching the domain and range of the set maps at each stage.   The orbit equivalence will be defined on the set of points for which our procedure will converge.  In order for this set to be a $G_\delta$ set, and the map to be a homeomorphism, for each point where we have convergence we need to have the procedure stabilize after a finite number of steps.  In other words, once we have defined the set map at a particular stage $n$, we cannot modify its domain at any successive stage.  

This introduces an obvious complication in the construction.  At a particular stage $n$, we have to know that our choice of the set map on stage $n$ towers will be consistent with 
the choices that we will make for all stages after.  To address this complication, informally speaking, we do not actually choose a particular set map at any stage.  Instead, at each stage we construct a collection of set maps that are possible extensions of previous stage set maps, and that all agree on a set we call {\it the good set}.  The convergence then depends on us being able to provide enough choices at each stage $n$ so that it is possible to construct sufficiently many choices of maps at stage $n+1$ that extend the $n$-stage maps.

Templates are a combinatorial tool that have been designed to facilitate the intensive book keeping required to describe such a procedure.  Formally, {\em template} is an ordered multiset.  For example, the multiset $\{a, a, b, b, c\}$, together with the ordering $a \prec a \prec b \prec c \prec b$, gives a template $\tau$, which we write as 
 $$
 \tau: a \prec a \prec b \prec c \prec b. 
 $$
 The elements of a template (with multiplicity) are called {\em levels}.    In our work, each level of a template will correspond to a clopen set that is a level of a tower.  In particular, the towers themselves can be thought of as templates where each level appears exactly once and the ordering on the levels is exactly the ordering on the sets that is imposed by the underlying dynamics.  
 
 Notice that a set map from one tower to another can be thought of as a re-ordering of the levels of the domain tower according to the levels that they are being mapped to in the image tower.  We replace the notion of a set map with maps between templates, where the re-ordering is not described by the map, but rather the ordering given by the image template. 

More formally, given templates $\tau$ and $\tau'$, written \[\tau: c_0 \prec c_1 \prec \cdots \prec c_{n-1} \hspace{1cm} \mbox{and} \hspace{1cm} \tau': d_0 \prec d_1 \prec \cdots \prec d_{m-1},\]  we {\em represent} $\tau$ and $\tau'$ with the intervals \[I = [0, 1, \ldots, n-1] \subset \z \hspace{1cm} \mbox{and} \hspace{1cm} J = [0, 1, \ldots, m-1] \subset \z\] via the correspondences $c_i \leftrightarrow i$ and $d_i \leftrightarrow i$.  A {\em partial interval bijection} is an ordered quintuple $\hat f = [I, J, A, B, f]$, where $A \subset I$, $B \subset J$, and $f: A \rightarrow B$ is some bijection. 

This perspective allows for an explicit combinatorial understanding of how many maps need to be defined at each stage in order to construct maps at later stages, on larger domains.  It also allows for an explicit combinatorial description of how maps from one stage are constructed from maps of a previous stage, so it is easy to prove that the construction has indeed stabilized for points on a $G_\delta$ set of full measure.

\subsection{The Induction}

In our proof we will construct an increasing sequence $(k_n)$, templates for each system and partial interval bijections between template sets that constitute the ``back and forth" diagram given in Figure \ref{thefigure} below.

\begin{figure} 
\centering
\scalebox{.5}{\includegraphics[width=1.25\textwidth]{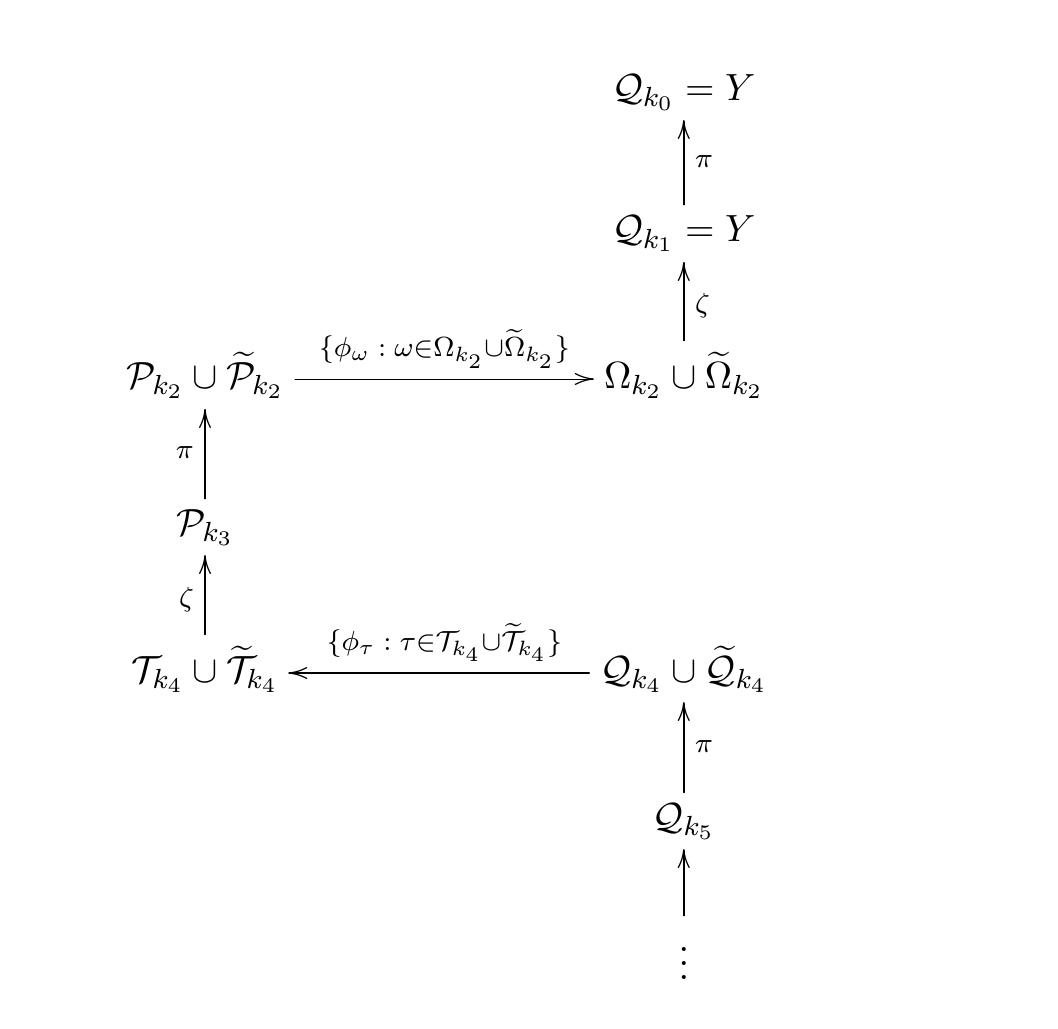}}
\caption{The ``back and forth" diagram \label{thefigure}}
\end{figure}


The objects in the diagram are template sets.  The template sets on the left ($\mathcal P_{k_n}$, $\widetilde{\mathcal P}_{k_n}$, $\mathcal T_{k_n}$, and $\widetilde{\mathcal T}_{k_n}$) belong to the Morse minimal system, while those on the right ($\mathcal Q_{k_n}$, $\widetilde{\mathcal Q}_{k_n}$, $\Omega_{k_n}$, and $\widetilde{\Omega}_{k_n}$) belong to the binary odometer.  The maps $\phi_*$ are partial interval bijections.  A key ingredient of the diagram is its {\it almost commutative} nature, as introduced by Roychowdhury and Rudolph in \cite{RR1}.  Interpreting the levels of templates as levels of towers that form a refining sequence of partitions, we see that every level at stage $n$ of a tower is a subset of a level from a previous stage template, and the maps $\zeta$ and $\pi$ are the natural inclusion maps.  The consistency of set maps from one level to another is achieved by requiring that on the good set, all partial interval bijections agree when composed with $\zeta$ and $\pi$. 

There are two key differences our work here and that of \cite{RR1} or \cite{DR} in how we use templates from a particular stage to construct later stage templates.  In particular, we cannot use the notion of concatenation as was defined in the earlier papers, instead we define {\it overlapping concatenations}.  In addition, to accommodate the combinatoric structure of the towers of the Morse minimal system we introduce a new family of partial interval bijections called {\it reordering maps}.    Once the diagram is built up, however, the argument that it produces a well-defined nearly continuous Kakutani equivalence is nearly identical to the arguments in both \cite{RR1} and \cite{DR}.   We include it here for completeness (see Sections \ref{convergenceofPIBs} and \ref{kakequivalence}).  

It is our hope that one day we might discover a  more general machine that could characterize broad classes of systems, perhaps even all zero entropy loosely Bernoulli systems.  But at the moment it is not clear how such a machine, if one exists, could be sufficiently general to account for the differences between systems. 

\subsection{The Organization of the Paper}

The paper is organized as follows:
 \begin{description}
\item[Section \ref{overview}] We define the tools for constructing templates and partial interval bijections that we will use throughout the construction.
\item[Sections \ref{morsepreliminaries} - \ref{binarypreliminaries}]  We give preliminary definitions of the Morse minimal system and the binary odometer.  In particular, we define the template sets $\mathcal P_{k}$, $\widetilde{\mathcal P}_{k}$, $\mathcal Q_{k}$, and $\widetilde{\mathcal Q}_{k}$, for $k \geq 0$.  These template sets are given by the towers in the respective system.
\item[Section \ref{omegaksection}] We define the template sets $\Omega_k$ and $\widetilde \Omega_k$, for $k \geq 0$. These are templates in the odometer system that are rearrangements of the tower templates, reflecting set maps that map Morse towers to odometer towers.
\item[Section \ref{stage2}] We construct stage $n = 2$ of the diagram by explicitly defining $k_2$, along with the partial interval bijections $\phi_\omega: \mathcal P_{k_2} \cup \widetilde{\mathcal P}_{k_2} \rightarrow \Omega_{k_2} \cup \widetilde{\Omega}_{k_2}$. 
\item[Section \ref{tksection}] We define the template sets $\mathcal T_k$ and $\widetilde{\mathcal T}_k$, for $k \geq 0$.  These are new templates in the Morse system, describing how odometer towers will be mapped to Morse towers.
\item[Section \ref{sequences}] We define the sequence $(k_n)$ recursively.
\item[Sections \ref{freqnotation} - \ref{stage8}] We assume the diagram has been built down to stage $n$, where $n$ is even, and show how to build it down to stage $n+2$.  Because the construction depends on whether $n \equiv 2$ or $n \equiv 0 \mod 4$, we proceed as follows: \begin{itemize}
\item  In Sections \ref{freqnotation} - \ref{blockpartitionsoftauhat}, we introduce notation and machinery that is used for every even $n$.
\item In Sections \ref{thegoodset} - \ref{stage4missingextra}, we illustrate the $n \equiv 2 \mod 4$ case by building the diagram down to stage $4$.
\item In Sections \ref{stage6bigpicture} - \ref{stage6missingextra}, we illustrate the $n \equiv 0 \mod 4$ case by building the diagram down to stage $6$.
\item In Section \ref{stage8}, we indicate how the induction looks in stages $n \geq 8$. 
\end{itemize}
\item[Sections \ref{convergenceofPIBs} - \ref{kakequivalence}] We use the properties of the diagram to prove that our procedure produces a nearly continuous Kakutani equivalence between the two systems.

\end{description}

\section{Partial Interval Bijections and Concatenations} \label{overview}

In this section we introduce the terminology and tools necessary to build and extend partial interval bijections.

\begin{definition} \label{domainandrange} \rm
Given templates $\tau$ and $\tau'$, and a partial interval bijection $\hat f = [I, J, A, B, f]$ from $\tau$ to $\tau'$, the {\em domain} of $\hat f$ consists of those levels in $\tau$ that are represented by $A$.  The {\em range} of $\hat f$ consists of those levels in $\tau'$ that are represented by $B$.
\end{definition} 

For the next definition, suppose $\tau'': e_0 \prec e_1 \prec \cdots \prec e_{m-1}$ is another template which, like $\tau'$, is represented by $J\subset \z$.  Suppose there is a partial interval bijection $\ds \hat{g} = [I, J, A, C, g]$ from $\tau$ to $\tau''$.  Note that $\ds \hat f$ and $\ds \hat{g}$ agree in their first three components ($I$, $J$, and $A$). 

\begin{definition} \label{match} \rm

Suppose $\hat f: \tau \rightarrow \tau'$ and $\ds \hat{g}: \tau \rightarrow \tau''$ are two partial interval bijections, given by $\hat f = [I, J, A, B, f]$ and $\hat g = [I, J, A, C, g]$.  Then $\hat f$ and $\hat g$ {\em match} if, for each integer $i \in A$, the level in $\tau'$ that is represented by $f(i)$ is identical to the level in $\tau''$ that is represented by $g(i)$.
\end{definition}

 \begin{definition} \rm
 Two partial interval bijections are {\em equivalent} if one is a translate of the other.  More precisely, $[I, J, A, B, f] \sim [I', J', A', B', f']$ if there exist $t, s \in \z$ and $I' = I + t$, $J' = J+s$, $A' = A+t$, $B' = B+s$, and $f'(i+t) = f(t) + s$.
 \end{definition}

\begin{definition} \rm
Given a partial interval bijection $\hat f = [I, J, A, B, f]$, the {\em inverse} of $\hat f$ is the partial interval bijection $\hat f^{-1} = [J, I, B, A, f^{-1}]$.
\end{definition}

\subsection{Simple Concatenations}

Suppose $\hat f_i = [I_i, J_i, A_i, B_i, f_i]$ for $i = 1, 2$ are two partial interval bijections, and assume each $I_i$ and $J_i$ begin at $0$.  Let $t = \#I_1$ and $s = \#J_1$.  Define the {\em simple concatenation} of $\hat f_1$ and $\hat f_2$ by \[\hat f_1 \ast \hat f_2 = [I_1 \cup (I_2 + t), J_1 \cup (J_2 + s), A_1 \cup (A_2 + t), B_1 \cup (B_2 + s), f],\] where \[f(i) = \begin{cases} f_1(i) & \mbox{ if }i \in A_1; \\ f_2(i-t)+s & \mbox{ if }i \in A_2+t \end{cases} \]  Note that this is an associative semigroup action on the space of all partial interval bijections.  Also note that, if $\hat f_1$ and $\hat f_2$ are partial interval bijections, then so is $\hat f_1 \ast \hat f_2$.

\subsection{Sticky Notes}  \label{stickynotes}

For our construction, some partial interval bijections will need to be decomposed into a form \begin{equation} \label{stage2PIBstructure} \hat f = \hat f(1) \ast \hat f(2) \ast \hat f(3),   \end{equation} where $\hat f(1)$ and $\hat f(3)$ make up a very small portion of the overall map.  







  
\noindent We will refer to $\hat f(1)$ and $\hat f(3)$ as the {\em bottom} and {\em top sticky notes} of $\hat f$.  The {\em body} is $\hat f(2)$.

\subsection{Overlapping Concatenations}

We will often want to ``glue" partial interval bijections together.  Top and bottom sticky notes will be our means of doing this via the following definition.

\begin{definition} \label{overlapconcatenation1} \rm

Suppose $\hat f_i = [I_i, J_i, A_i, B_i, f_i]$ for $i = 1, 2$ are two partial interval bijections that are decomposed via (\ref{stage2PIBstructure}) as  \[\hat f_i = \hat f_i(1) \ast \hat f_i(2) \ast \hat f_i(3).\]   If $\hat f_1(3) \sim \hat f_2(1)$, then define the {\em overlapping concatenation} of $\hat f_1$ and $\hat f_2$, denoted $\ds \hat f_1 \tilde \ast \hat f_2$, by \begin{eqnarray*} \hat f_1 \tilde \ast \hat f_2 & = & \hat f_1(1) \ast \hat f_1(2) \ast \hat f_1(3) \ast \hat f_2(2) \ast \hat f_2(3) \\ & = & \hat f_1(1) \ast \hat f_1(2) \ast \hat f_2(1) \ast \hat f_2(2) \ast \hat f_2(3) \end{eqnarray*}

\end{definition} 








\subsection{Generalized Sticky Notes}

The sticky notes described in Section \ref{stickynotes} will be used only in the early stages of the construction.  From then on, top and bottom sticky notes will {\em overlap} with the body, so that a typical partial interval bijection will need to be decomposed into the form    \begin{equation} \label{stage4PIBstructure} \hat f = \hat f(1) \tilde \ast \hat f(2) \tilde \ast \hat f(3),   \end{equation} where again $\hat f(1)$ and $\hat f(3)$ make up a very small portion of the overall map.  Here again, we will refer to $\hat f(1)$ and $\hat f(3)$ as the bottom and top sticky notes, and $\hat f(2)$ as the body, of $\hat f$.

\subsection{Generalized Overlapping Concatenations}

Once we are far enough along in the construction that sticky note decompositions take the form (\ref{stage4PIBstructure}), we will no longer be able to use Definition \ref{overlapconcatenation1} to glue partial interval bijections together.  We will instead use the following.

\begin{definition} \label{overlapconcatenation2} \rm

Suppose $\hat f_i = [I_i, J_i, A_i, B_i, f_i]$ for $i = 1, 2$ are two partial interval bijections that are decomposed via (\ref{stage4PIBstructure}) as  \[\hat f_i = \hat f_i(1) \tilde \ast \hat f_i(2) \tilde \ast \hat f_i(3).\]   If $\hat f_1(3) \sim \hat f_2(1)$, then define the {\em overlapping concatenation} of $\hat f_1$ and $\hat f_2$, denoted $\ds \hat f_1 \tilde \ast \hat f_2$, by \begin{eqnarray*} \hat f_1 \tilde \ast \hat f_2 & = & \hat f_1(1) \tilde \ast \hat f_1(2) \tilde \ast \hat f_1(3) \tilde \ast \hat f_2(2) \tilde \ast \hat f_2(3) \\ & = & \hat f_1(1) \tilde \ast \hat f_1(2) \tilde \ast \hat f_2(1) \tilde \ast \hat f_2(2) \tilde \ast \hat f_2(3) \end{eqnarray*}

\end{definition} 

\subsection{Reordering Maps} \label{reorderingmaps}

A {\em reordering map} is a partial interval bijection of the form $\hat p = [J, J, J, J, p]$.  Unlike arbitrary partial interval bijections, in a reordering map each of the first four components is the same interval.  Therefore it is possible to compose two reordering maps, as follows.  If $p_1: J \rightarrow J$ and $p_2: J \rightarrow J$ are two bijections, then, as a usual composition of functions, $p_2 \circ p_1: J \rightarrow J$ is a bijection.  Therefore we can define $\hat p_2$ {\em composed with} $\hat p_1$ to be the reordering map $\hat p_2 \circ \hat p_1 = [J, J, J, J, p_2 \circ p_1]$.  

Reordering maps will be used to move certain levels in the ``bottom part" of a template up to the ``top part," while shifting all levels in the ``middle part" down.  For example, consider the template \[\tau: c_0 \prec c_1 \prec c_2 \prec c_3 \prec c_4 \prec c_5 \prec c_6 \prec\mbox{\fbox{$c_7 \prec c_8 \prec c_9$}}  \prec c_{10} \prec c_{11} \prec c_{12} \prec c_{13},\] and think of $c_7 \prec c_8 \prec c_9$ as the ``middle part".  Suppose we wish to shift this middle part down by two positions.  We could accomplish this, for example, by moving $c_2$ and $c_6$ from the bottom part to the top part, as follows.   Let $J = [0, 1, \ldots, 13]$ and define $p_1: J \rightarrow J$ by \[p_1(j) = \begin{cases} 9 & \mbox{ if } j = 2 \\ j - 1 & \mbox{ if } 3 \leq j \leq 9 \\ j & \mbox{ if } 0 \leq j \leq 1 \mbox{ or } 10 \leq j \leq 13 \end{cases} \] and $p_2 : J \rightarrow J$ by  \[p_2(j) = \begin{cases} 11 & \mbox{ if } j = 5 \\ j - 1 & \mbox{ if } 6 \leq j \leq 11 \\ j & \mbox{ if } 0 \leq j \leq 4 \mbox{ or } 12 \leq j \leq 13. \end{cases} \] Notice that $\hat p_1(\tau)$ is the (new) template \[c_0 \prec c_1 \prec c_3 \prec c_4 \prec c_5 \prec c_6 \prec \mbox{\fbox{$c_7 \prec c_8 \prec c_9$}} \prec c_2 \prec c_{10} \prec c_{11} \prec c_{12} \prec c_{13},\] and $\hat p_2 \circ \hat p_1(\tau)$ is the (new) template \[c_0 \prec c_1 \prec c_3 \prec c_4 \prec c_5 \prec \mbox{\fbox{$c_7 \prec c_8 \prec c_9$}} \prec c_2 \prec c_{10} \prec c_{11} \prec c_6 \prec c_{12} \prec c_{13}.\]     

We can also compose a reordering map $\hat p = [J, J, J, J, p]$ with an arbitrary partial interval bijection $\hat f = [I, J, A, B, f]$ by defining $\ds \hat p \circ \hat f = [I, J, A, p(B), p|_B\circ f]$, where $p|_B$ is the restriction of $p$ to $B$.

\section{Morse Minimal System Preliminaries} \label{morsepreliminaries}

Given a binary word $B = b_1b_2\cdots b_k \subset \{0, 1\}^k$ of length $k$, the {\em flip} of $B$ is the word $\overline B = \overline b_1 \overline b_2 \cdots \overline b_k$, where $\overline b_i = 0$ if $b_i = 1$ and $\overline b_i = 1$ if $b_i = 0$. 

Let $\sigma$ be the substitution rule on the symbols $0$ and $1$ given by $ \sigma(0)  =  01$ and $\sigma(1)  =  10$.  Iterating $\sigma$ on the symbol $0$ determines, for each $k \in \n$, a word $u_k := \sigma^k(0)$ of length $2^k$: \begin{eqnarray*} u_1 & = & \sigma(0)  \; = \;   01 \\ u_2 & = & \sigma^2(0)  \; = \; 0110 \\ u_3 & = & \sigma^3(0)  \; = \; 01101001 \\ &  & \vdots  \end{eqnarray*}  Observe that $\overline u_k = \sigma^k(1)$ and that $u_{k+1} = u_k\overline u_k$.  The {\em Morse sequence} is the sequence in $\{0, 1\}^\n$ whose first $2^k$ symbols are the word $u_k$.  

Let $X$ denote the set of all doubly infinite sequences $x = (x_i)_{i \in \z}$ in $\{0, 1\}^\z$ such that every finite subword of $x$ occurs as a subword of the Morse sequence.  Given $x$ and $x'$ in $X$, let $\rho(x, x') = 1$ if $x_0 \neq x_0'$; otherwise, define $\ds \rho(x, x') = \frac{1}{2^n}$, where $n$ is maximal such that $x_i = x_i'$ for all $i \in \{-n, \ldots, n\}$.  Then $\rho$ is a metric on $X$ that determines a Borel sigma algebra $\mathcal B$.  The {\em Morse minimal system} is then the system $(X, T, \mathcal B, \mu)$, where $T: X \rightarrow X$ is the left shift, and $\mu$ is the unique complete ergodic Borel probability measure.

The following results are proved in \cite{RR3}.

\begin{proposition}[\cite{RR3}] \label{partitionprop} For each $x \in X$ and $k \in \n$, there exists a unique partition of $\z$ into intervals of length $2^{k+1}$ so that the subword of $x$ on each interval of this partition is either $u_k \overline u_k$ or $\overline u_k u_k$.
\end{proposition}

\begin{corollary}[\cite{RR3}] \label{partitioncor} For each $x \in X$ and $k \in \n$, there exists a unique partition of $\z$ into intervals of length $2^k$ so that the subword of $x$ on each interval of this partition is either $u_k$ or $\overline u_k$.
\end{corollary}

\begin{proposition}[\cite{RR3}] \label{0coordinate}For $x \in X$ and $k \in \n$, partition $\z$ into intervals as in Corollary \ref{partitioncor} and let $t_k(x)$ be the position occupied by $0$ in its interval, $0 \leq t_k(x) < 2^k$.  The functions $t_k$ are continuous.  
\end{proposition}

\subsection{Canonical Templates} \label{cantowermorse} 

Given $x \in X$ and $k \in \n$, let $t_k(x)$ be as defined in Proposition \ref{0coordinate}.  For $0 \leq i < 2^k$, let $u_{k, i}$ denote the $i$th symbol in $u_k$, and let $\overline u_{k, i}$ denote the $i$th symbol in $\overline u_{k}$.  Define sets  \[u_k(i) = \{x \in X \; | \; t_k(x) = i  \mbox{ and } x_0 = u_{k, i}\}\] and \[\overline u_k(i) = \{x \in X \; | \; t_k(x) = i \mbox{ and } x_0 = \overline u_{k, i}\}.\] Then $u_k(i)$ and $\overline u_k(i)$ are obtained by taking the clopen set where $t_k(x) = i$ and splitting it according to whether the symbol at the origin is $0$ or $1$.  Therefore $u_k(i)$ and $\overline u_k(i)$ are clopen.  These $2^{k+1}$ sets, all of equal measure, are called the {\em $k$-canonical cylinders} in $X$.  We will often refer to $k$-canonical cylinders as {\em levels}.

Define the {\em k-canonical templates} in $X$ by:   \[ \mathcal P_k(0): u_k(0) \prec u_k(1) \prec \cdots \prec u_k(2^k-1) \] and \[\mathcal P_k(1): \overline u_k(0) \prec \overline u_k(1) \prec \cdots \prec \overline u_k(2^k-1).\]   Note that each $k$-canonical template has height $2^k$.  The order on the $k$-canonical templates is simply the order given by the action of $T$.   The set of levels in the $k$-canonical templates gives a partition of $X$.  Let $\mathcal P_k = \{\mathcal P_k(0), \mathcal P_k(1)\}$.




\subsection{Maps Between Canonical Templates} \label{Pkpi} Given $0 \leq k' < k$, each level $c \in \mathcal P_{k}$ is a subset of a unique level $c' \in \mathcal P_{k'}$.  Sending $c \mapsto c'$ then gives a map $\pi: \mathcal P_k \rightarrow \mathcal P_{k'}$.  Observe that $\pi$ is measure preserving in the sense that the measure of the pull back of a set in $\mathcal P_{k'}$ is the same as its measure.

\subsection{Modified Canonical Templates}  Define three modified versions of each $k$-canonical template: one in which the bottom level is removed, a second in which there is an extra copy of $u_k(0)$ tacked on to the top, and a third in which there is an extra copy of $\overline u_k(0)$ tacked on to the top.  The superscripts $m$, $e(0)$, and $e(1)$ will denote ``missing bottom level", ``extra copy of $u_k(0)$", and ``extra copy of $\overline u_k(0)$", respectively:

\begin{eqnarray*} \mathcal P_k(0)&:& u_k(0) \prec u_k(1) \prec \cdots \prec u_k(2^k-1) \\
\mathcal P_k^m(0)&:&u_k(1) \prec u_k(2) \prec \cdots \prec u_k(2^k-1) \\
\mathcal P_k^{e(0)}(0)&:& u_k(0) \prec u_k(1) \prec \cdots \prec u_k(2^k-1) \prec u_k(0) \\
\mathcal P_k^{e(1)}(0)&:& u_k(0) \prec u_k(1) \prec \cdots \prec u_k(2^k-1) \prec \overline u_k(0) \end{eqnarray*} 

\noindent Define $\ds \mathcal P_k^m(1), \mathcal P_k^{e(0)}(1)$, and $\ds \mathcal P_k^{e(1)}(1)$ by analogy.  Let
 \begin{equation}\label{modtem}
 \widetilde{\mathcal P}_k = \{\mathcal P_k^m(0), \mathcal P_k^{e(0)}(0), \mathcal P_k^{e(1)}(0), \mathcal P_k^m(1), \mathcal P_k^{e(0)}(1), \mathcal P_k^{e(1)}(1)\}
 \end{equation}
  be the set of all modified canonical templates at stage $k$.

\section{Binary Odometer Preliminaries} \label{binarypreliminaries}

Let $Y = \{0, 1\}^{\mathbb N}$.  Then $Y$ is compact and metrizable; the metric $\rho(y, y') = 2^{-\ell}$, where $\ell = \mbox{min}\{ |i| : y_i \neq y_i' \}$, induces the topology and determines a Borel sigma algebra $\mathcal F$.  Define $S: Y \rightarrow Y$ by $S(y) = y + {\mathbf 1}$, where $\mathbf 1 = (1, 0, 0, \ldots)$ and the addition is coordinate-wise mod $2$ with right carry.  The {\em binary odometer} is the system $(Y, S, \mathcal F, \nu)$, where $\nu$ is the unique complete ergodic Borel probability measure.

\subsection{Canonical Templates} \label{canonicalbinary} A {\em $k$-canonical cylinder} is a set $\{y \in Y \; | \; y \mbox{ begins with }d\}$, where $d\in \{0, 1\}^k$ is a binary word of length $k$.  These $k$-canonical cylinders are the levels of templates for the binary odometer.  

The {\em k-canonical template} in $Y$, denoted $\mathcal Q_k(0)$, is the set of $k$-canonical cylinders together with the order $\prec$ inherited by the action of $S$, where $\mathbf 0 = (0, 0, 0, \ldots)$ is an element of the first cylinder.  For example, the $3$-canonical template is  \begin{equation} \label{tpexample}\mathcal Q_3(0): 000 \prec 100 \prec 010 \prec 110 \prec 001 \prec 101 \prec 011 \prec 111.\end{equation}  In general, let $v_k(i)$ denote the $i$-th level in $\mathcal Q_k(0)$, so that  \[\mathcal Q_k(0): v_k(0) \prec v_k(1) \prec \cdots \prec v_k(2^k-1).\]  Note that $\mathcal Q_k(0)$ has height $2^k$, and the set of levels in $\mathcal Q_k(0)$ gives a partition of $Y$.  Let $\mathcal Q_k = \{\mathcal Q_k(0)\}$.



\subsection{Maps Between Canonical Templates} \label{Qkpi} Given $0 \leq k' < k$, each level $d \in \mathcal Q_k$ is a subset of a unique level $d' \in \mathcal Q_{k'}$.  Sending $d \mapsto d'$ then gives a map $\pi: \mathcal Q_k \rightarrow \mathcal Q_{k'}$.  Observe that $\pi$ is measure preserving in the sense that the measure of the pull back of a set in $\mathcal Q_{k'}$ is the same as its measure.

\subsection{Modified Canonical Templates}  Define two modified versions of the $k$-canonical template $\mathcal Q_k(0)$: one in which the level $v_k(0)$ is removed, and another in which there is an extra copy of $v_k(0)$ tacked on to the top.  The superscripts $m$ and $e$ will denote ``missing $v_k(0)$" and ``extra copy of $v_k(0)$", respectively:

\begin{eqnarray*}
\mathcal Q_k(0) & : & v_k(0) \prec v_k(1) \prec \cdots \prec v_k(2^k-1) \\
\mathcal Q_k^m(0) & : & v_k(1) \prec v_k(2) \prec \cdots \prec v_k(2^k-1)  \\
\mathcal Q_k^e(0) & : & v_k(0) \prec v_k(1) \prec \cdots \prec v_k(2^k-1) \prec v_k(0) \end{eqnarray*}

\noindent Let \[\widetilde{\mathcal Q}_k = \{\mathcal Q_k^m(0), \mathcal Q_k^e(0)\}\] be the set of all modified canonical templates at stage $k$.
Define $k_0 = k_1 = 0$, so that both $\mathcal Q_{k_0}$ and $\mathcal Q_{k_1}$ are the trivial canonical templates for the binary odometer (each consisting of just one level)
as indicated in Figure~\ref{thefigure}.

\section{Binary Odometer Template Sets $\Omega_k$ and $\widetilde{\Omega}_k$} \label{omegaksection}

The template sets $\Omega_k$ and $\widetilde{\Omega}_k$, defined in Section~\ref{omegakdefinition} 
consist of templates of the following types: basic, diminished, augmented, missing and extra.  We begin by describing these types of templates.  

\subsection{Basic Templates}

A {\em basic template} at stage $k$ is any template that satisfies all of the following conditions: it has height $2^k$; its elements are $k$-canonical cylinders; and
its ordering is ``allowed" by the action of $S$.

For the ordering of a template $\omega$ to be ``allowed" by the action of $S$, $\omega$ must have one of the following forms.  Either it is a special basic template
$\omega = \mathcal Q_k(0)$  that we will refer to as the {\em zero-template}, or 
for some (unique) $i \in \{1, 2, \ldots, 2^k-1\}$, $\omega$ is:
 \[\omega: v_k(i) \prec v_k(i+1) \prec \cdots \prec v_k(2^k-1) \prec v_k(0) \prec v_k(1) \prec \cdots \prec v_k(i-1).\]

\noindent Let $\mathcal B_k(Y)$ be the set of all basic templates for the binary odometer at stage $k$.


In each basic template, the level $v_k(0)$ occurs once and only once.  Call this level the {\em global cut}.

\subsection{Predecessor and Successor Templates}

If $\omega \in \mathcal B_k(Y)$ is a basic template, define the {\em predecessor} template for $\tau$, denoted $\omega_p$, to be the basic template whose global cut is one position higher (mod $2^k$), and define the {\em successor} template for $\omega$, denoted $\omega_s$, to be the basic template whose global cut is one position lower (mod $2^k$).  For example, if $\omega \in \mathcal B_k(Y)$ is: \[\omega: v_k(i) \prec v_k(i+1) \prec \cdots \prec v_k(2^k-1) \prec v_k(0) \prec v_k(1) \prec \cdots \prec v_k(i-1),\] then $\omega_p$ is \[\omega_p: v_k(i-1) \prec v_k(i) \prec \cdots \prec v_k(2^k-1) \prec v_k(0) \prec v_k(1) \prec \cdots \prec v_k(i-2),\] and $\omega_s$ is \[\omega_s: v_k(i+1) \prec v_k(i+2) \prec \cdots \prec v_k(2^k-1) \prec v_k(0) \prec v_k(1) \prec \cdots \prec v_k(i).\] 

\subsection{Diminished Templates}

Given a basic template $\omega \in \mathcal B_k(Y)$, define two additional, {\em diminished} templates $\omega^-(d)$ and $\omega^-(u)$ as follows:

\begin{itemize}
\item $\omega^-(d)$ is $\omega$ with the global cut removed, and with an extra level tacked on at the bottom (the level that would naturally precede the bottom level in $\omega$, namely the pre-image of the bottom level of $\omega$ under the map $S$).

\item $\omega^-(u)$ is $\omega$ with the global cut removed, and with an extra level tacked on at the top (the level that would naturally follow the top level in $\omega$, namely the image of the top level in $\omega$ under the map $S$).
\end{itemize}

\noindent For example, if $\omega \in \mathcal B_k(Y)$ is given by \[\omega: v_k(i) \prec v_k(i+1) \prec \cdots \prec v_k(2^k-1) \prec v_k(0) \prec v_k(1) \prec \cdots \prec v_k(i-1),\] then $\omega^-(d)$ is given by \[\omega^-(d): v_k(i-1) \prec v_k(i) \prec \cdots \prec v_k(2^k-1) \prec v_k(1) \prec \cdots \prec v_k(i-1).\]

\noindent The letters ``d" and ``u" refer to ``down" and ``up", respectively, for reasons that will be made clear later.

 Let $\mathcal D_k(Y)$ be the set of all diminished templates for the binary odometer at stage $k$.  Note that all diminished templates in $\mathcal D_k(Y)$ have height $2^k$.

\subsection{Augmented Templates}

Given a basic template $\omega \in \mathcal B_k(Y)$, define two additional, {\em augmented} templates as follows:

\begin{itemize}
\item $\omega^+(d)$ is $\omega$ with an extra copy of the global cut, $v_k(0)$, inserted right next to the actual global cut, and with the bottom level deleted.

\item $\omega^+(u)$ is $\omega$ with an extra copy of the global cut, $v_k(0)$, inserted right next to the actual global cut, and with the top level deleted.

\end{itemize}

\noindent For example, if $\omega \in \mathcal B_k(Y)$ is given by \[\omega: v_k(i) \prec v_k(i+1) \prec \cdots \prec v_k(2^k-1) \prec v_k(0) \prec v_k(1) \prec \cdots \prec v_k(i-1),\] then $\omega^+(d)$ is given by \[\omega^+(d): v_k(i+1) \prec \cdots \prec v_k(2^k-1) \prec v_k(0) \prec v_k(0) \prec v_k(1) \prec \cdots \prec v_k(i-1).\] 

\noindent Let $\mathcal A_k(Y)$ be the set of all augmented templates for the binary odometer at stage $k$.  Note that all augmented templates in $\mathcal A_k(Y)$ have height $2^k$.

\subsection{Missing Templates}

Given a basic template $\omega \in \mathcal B_k(Y)$, define one additional, {\em missing} template, denoted $\omega^m$, to be $\omega$ with its bottom level removed.  For example, if $\omega$ is the zero-template, then \[\omega^m = v_k(1) \prec v_k(2) \prec \cdots \prec v_k(2^k-1).\]

\noindent Let $\mathcal M_k(Y)$ be the set of all missing templates for the binary odometer at stage $k$.  Note that all missing templates have height $2^k-1$.

\subsection{Extra Templates}

Given a basic template $\omega \in \mathcal B_k(Y)$, define one additional, {\em extra} template, denoted $\omega^e$, to be $\omega$ with an one extra level tacked on at the top (the level that would naturally follow the top level).  For example, if $\omega \in \mathcal B_k(Y)$ is given by \[\omega: v_k(i) \prec v_k(i+1) \prec \cdots \prec v_k(2^k-1) \prec v_k(0) \prec v_k(1) \prec \cdots \prec v_k(i-1),\]  then $\omega^e$ is given by \[\omega^e: v_k(i) \prec v_k(i+1) \prec \cdots \prec v_k(2^k-1) \prec v_k(0) \prec v_k(1) \prec \cdots \prec v_k(i-1)\prec v_k(i).\] 

\noindent Let $\mathcal E_k(Y)$ be the set of all extra templates for the binary odometer at stage $k$.  Note that all extra templates in $\mathcal E_k(Y)$ have height $2^k+1$.

\subsection{Definition of $\Omega_k$ and $\widetilde \Omega_k$} \label{omegakdefinition}

Define
\begin{equation}\label{modom}
\Omega_k = \mathcal B_k(Y) \cup \mathcal D_k(Y) \cup \mathcal A_k(Y)\text{ and } \widetilde \Omega_k = \mathcal M_k(Y) \cup \mathcal E_k(Y).
\end{equation}

\section{Second stage of the induction} \label{stage2}

Define $k_2 = 2$, so that the canonical templates in $\mathcal P_{k_2}$ have height $4$.  For each canonical template $\mathcal P_{k_2}(i) \in \mathcal P_{k_2}$ for the Morse system, and for each template $\omega \in \Omega_{k_2}$ in the binary odometer, we will define a partial interval bijection $\phi_\omega: \mathcal P_{k_2}(i) \rightarrow \omega$ from a subset of the levels in $\mathcal P_{k_2}(i)$ to a subset of the levels in $\omega$.  Only after we have done this will we define partial interval bijections from the modified canonical templates in $\widetilde{\mathcal P}_{k_2}$ to the modified templates in $\widetilde{\Omega}_{k_2}$.

\subsection{Maps to Basic Templates}  There are exactly four basic templates in $\mathcal B_{k_2}(Y)$: 

\begin{itemize}
\item $\omega_1: v_{k_2}(0) \prec v_{k_2}(1) \prec v_{k_2}(2) \prec v_{k_2}(3)$, 
\item $\omega_2: v_{k_2}(3) \prec v_{k_2}(0) \prec v_{k_2}(1) \prec v_{k_2}(2)$, 
\item $\omega_3: v_{k_2}(2) \prec v_{k_2}(3) \prec v_{k_2}(0) \prec v_{k_2}(1)$,  and 
\item $\omega_4: v_{k_2}(1) \prec v_{k_2}(2) \prec v_{k_2}(3) \prec v_{k_2}(0)$.
\end{itemize}

\noindent For $0 \leq i \leq 1$ and $1 \leq j \leq 4$, define $\phi_{\omega_j}: \mathcal P_{k_2}(i) \rightarrow \omega_j$ to be the partial interval bijection $\phi_{\omega_j} = [I, I, A, B_j, f_j]$, where:

\begin{itemize}
\item $I = [0, 1, 2, 3] \subset \z$,
\item $A = \{2\}$,
\item $B_1 = B_2 = \{2\}$,
\item $B_3 = B_4 = \{1\}$, and
\item Each $f_j$ is the obvious bijection  $f_j: A \rightarrow B_j$.
\end{itemize}  

\noindent Note that each $\phi_{\omega_j}(0)$ is equivalent to $\phi_{\omega_j}(1)$ as a formal map between intervals in $\z$.  But of course $\phi_{\omega_j(0)}$ and $\phi_{\omega_j(1)}$ are different as set maps because the levels represented by $A$ in $\mathcal P_{k_2}(0)$ and $\mathcal P_{k_2}(1)$ are different.   Now recall Definition \ref{domainandrange}, where the domain and range of a partial interval bijection are defined.  The following propositions are obvious.

\begin{proposition} \label{nobottomortop}
For each canonical template $\mathcal P_{k_2}(i) \in \mathcal P_{k_2}$, neither the bottom level nor the top level is in the domain of any $\phi_\omega$. 
\end{proposition}

\begin{proposition} \label{nobottomtoporcut}
Given $\omega \in \mathcal B_{k_2}(Y)$, the global cut in $\omega$ is not in the range of $\phi_\omega$.

\end{proposition}

\noindent We now define the ``good set" at stage $2$, and establish its most important property (Proposition \ref{stage2goodsetprop}), which is obvious at this stage because $\mathcal Q_{k_0}$ is the trivial canonical template.

\begin{definition}
Let $\mathcal G_2 = \{u_{k_2}(2), \overline u_{k_2}(2)\}$. 

\end{definition}

Recall the vertical map $\pi: \mathcal Q_{k_{1}} \rightarrow \mathcal Q_{k_0}$ which was defined in Section \ref{Pkpi}.  We similarly define $\zeta: \Omega_{k_2} \rightarrow \mathcal Q_{k_1}$.  The following is immediate since $\mathcal Q_{k_0} = \mathcal Q_{k_1} = Y$.

\begin{proposition} \label{stage2goodsetprop}
If $c \in \mathcal G_2$ and $\omega, \omega' \in \mathcal B_{k_2}(Y)$, then $c$ is in the domain of both $\phi_\omega$ and $\phi_{\omega'}$, and $\pi \circ \zeta \circ \phi_{\omega}(c) = \pi \circ \zeta \circ \phi_{\omega'}(c)$.
\end{proposition}

\subsection{Maps to Diminished Templates} Let $\omega$ be a basic template in $\mathcal B_{k_2}(Y)$.  Define $\phi_{\omega^-(u)}$ to match with $\phi_\omega$, and define $\phi_{\omega^-(d)}$ to match with $\phi_{\omega_p}$, as in Definition \ref{match}.  That these bijections are well defined follows from Proposition \ref{nobottomtoporcut}; for example, suppose $\omega$ is given by \[\omega: v_{k_2}(3) \prec v_{k_2}(0) \prec \mbox{\fbox{$v_{k_2}(1)$}} \prec v_{k_2}(2),\] where the box indicates the level that is in the range of $\phi_\omega$.  Then, using this same ``box" notation to indicate the levels in the ranges of the corresponding partial interval bijections, we have \[\omega^-(u): v_{k_2}(3) \prec \mbox{\fbox{$v_{k_2}(1)$}} \prec v_{k_2}(2) \prec v_{k_2}(3),\] \[\omega^-(d) : v_{k_2}(2) \prec \mbox{\fbox{$v_{k_2}(3)$}} \prec v_{k_2}(1) \prec v_{k_2}(2), \] and \[\omega_p : v_{k_2}(2) \prec \mbox{\fbox{$v_{k_2}(3)$}} \prec v_{k_2}(0) \prec v_{k_2}(1).\]  Now observe that $\omega_p^-(u)$, and its corresponding partial interval bijection, are: \[\omega_p^-(u) : v_{k_2}(2) \prec \mbox{\fbox{$v_{k_2}(3)$}} \prec v_{k_2}(1) \prec v_{k_2}(2).\]  That $\phi_{\omega_p^-(d)}$ matches with $\phi_{\omega^-(d)}$ is not a coincidence:

\begin{lemma} \label{stage2diminishedmatching}
Given any basic template $\omega \in \mathcal B_{k_2}(Y)$, $\phi_{\omega^-(d)}$ matches with $\phi_{\omega_p^-(u)}$.

\end{lemma}

\begin{proof}
Both $\phi_{\omega^-(d)}$ and $\phi_{\omega_p^-(u)}$ are defined to match with $\phi_{\omega_p}$.

\end{proof}

\subsection{Maps to Augmented Templates}  Similarly, given $\omega \in \mathcal B_{k_2}(Y)$, define $\phi_{\omega^+(d)}$ to match with $\phi_\omega$, and define $\phi_{\omega^+(u)}$ to match with $\phi_{\omega_p}$.  That these definitions are possible again follows from Proposition~\ref{nobottomtoporcut}.  

\begin{lemma} \label{stage2augmentedmatching}
Given any basic template $\omega \in \mathcal B_{k_2}(Y)$, $\phi_{\omega^+(u)}$ matches with $\phi_{\omega_p^+(d)}$.

\end{lemma}

\begin{proof}
Both $\phi_{\omega^+(u)}$ and $\phi_{\omega_p^+(d)}$ are defined to match with $\phi_{\omega_p}$.
\end{proof}

\subsection{Maps to Missing and Extra Templates}

Recall the modified template sets $\widetilde{\mathcal P}_{k_2}$ and $\widetilde \Omega_{k_2}$ defined in \eqref{modtem}  and \eqref{modom}.
Given $i \in \{0, 1\}$ and $\omega^m \in \mathcal M_{k_2}(Y)$, we know from Propositions \ref{nobottomortop} and \ref{nobottomtoporcut} that the bottom level of $\mathcal P_{k_2}(i)$ is not in the domain of $\phi_\omega : \mathcal P_{k_2}(i) \rightarrow \omega$, nor is the bottom level of $\omega$ in the range.  Therefore, if $\phi_\omega = [I, I, A, B, f]$, where $I = [0, 1, \ldots, 2^{k_2}-1]$, then we may define $\phi_{\omega^m} : \mathcal P_{k_2}^m(i) \rightarrow \omega^m$ by $\phi_{\omega^m} = [I', I', A, B, f]$, where $I' = [1, 2, \ldots, 2^{k_2}-1]$.  In other words, $\phi_{\omega^m}: \mathcal P_{k_2}^m(i) \rightarrow \omega^m$ is identical to $\phi_{\omega} : \mathcal P_{k_2}(i) \rightarrow \omega$ except that the bottom level of $\omega$ is technically missing.

Similarly, given $i, j \in \{0, 1\}$ and $\omega^e \in \mathcal E_{k_2}(Y)$, if again $\phi_\omega = [I, I, A, B, f]$, then we may define $\phi_{\omega^e}: \mathcal P_{k_2}^{e(j)}(i) \rightarrow \omega^e$ by $\phi_{\omega^e} = [I'', I'', A, B, f]$, where $I'' = [0, 1, \ldots, 2^{k_2}]$.  In other words, $\phi_{\omega^e}: \mathcal P_{k_2}^{e(j)}(i) \rightarrow \omega^e$ is identical to $\phi_\omega: \mathcal P_{k_2}(i) \rightarrow \omega$ except that there is technically an extra level at the top.

\section{Morse system template sets $\mathcal T_k$ and $\widetilde{\mathcal T}_k$} \label{tksection}

Similar to $\Omega_k$ and $\widetilde{\Omega}_k$, the template sets $\mathcal T_k$ and $\widetilde{\mathcal T}_k$, defined in Section~\ref{tkdefinition} below, consist of templates of the following types: basic, diminished, augmented, missing, and extra.  We define these types here.

\subsection{Basic Templates} \label{morsebasictemplates} A {\em basic template} at stage $k$ is any template that satisfies the all of the following conditions: it has height $2^k$, its elements are $k$-canonical cylinders, and its ordering is ``allowed" by the action of $T$.

For the ordering of a template $\tau$ to be ``allowed" by the action of $T$, $\tau$ must have one of the following forms:
\begin{enumerate}
\item $\tau = \mathcal P_k(0)$.  This is a special basic template we call the {\em zero-template}.
\item $\tau = \mathcal P_k(1)$.  This is a special basic template we call the {\em one-template}.
\item \label{fourforms} For some (unique) $i \in \{1, 2, \ldots, 2^k-1\}$, $\tau$ is one of the following four templates:  \[\tau: u_k(i) \prec u_k(i+1) \prec \cdots \prec u_k(2^k-1) \prec u_k(0) \prec u_k(1) \prec \cdots \prec u_k(i-1)\]
 \[\tau: u_k(i) \prec u_k(i+1) \prec \cdots \prec u_k(2^k-1) \prec \overline u_k(0) \prec \overline u_k(1) \prec \cdots \prec \overline u_k(i-1)\]
 \[\tau: \overline u_k(i) \prec \overline u_k(i+1) \prec \cdots \prec \overline u_k(2^k-1) \prec u_k(0) \prec u_k(1) \prec \cdots \prec u_k(i-1)\]
 \[\tau: \overline u_k(i) \prec \overline u_k(i+1) \prec \cdots \prec \overline u_k(2^k-1) \prec \overline u_k(0) \prec \overline u_k(1) \prec \cdots \prec \overline u_k(i-1)\]
\end{enumerate}

\noindent Let $\mathcal B_k(X)$ be the set of all basic templates for the Morse system at stage $k$.


In each basic template, exactly one of the levels is either $u_k(0)$ or $\overline u_k(0)$.  Call that level the {\em global cut}.  

\subsection{Predecessor and Successor Templates}

If $\tau \in \mathcal B_k(X)$ is a basic template whose global cut is neither the bottom level nor the top level, then $\tau$ must be one of the four templates listed in item \ref{fourforms} of Section \ref{morsebasictemplates} above.  In this case, define the {\em predecessor} template for $\tau$, denoted $\tau_p$, to be the basic template of that same form whose global cut is one position higher, and define the {\em successor} template for $\tau$, denoted $\tau_s$, to be the basic template of that same form whose global cut is one position lower.  For example, if $\tau \in \mathcal B_k(X)$ is the template  \[\tau: u_k(i) \prec u_k(i+1) \prec \cdots \prec u_k(2^k-1) \prec \overline u_k(0) \prec \overline u_k(1) \prec \cdots \prec \overline u_k(i-1),\] then $\tau_p \in \mathcal B_k(X)$ is  \[\tau_p: u_k(i-1) \prec u_k(i) \prec \cdots \prec u_k(2^k-1) \prec \overline u_k(0) \prec \overline u_k(1) \prec \cdots \prec \overline u_k(i-2),\] and $\tau_s \in \mathcal B_k(X)$ is \[\tau_s: u_k(i+1) \prec u_k(i+2) \prec \cdots \prec u_k(2^k-1) \prec \overline u_k(0) \prec \overline u_k(1) \prec \cdots \prec \overline u_k(i).\] 

If $\tau \in \mathcal B_k(X)$ is a basic template whose global cut is the bottom level, then $\tau$ is either the zero-template or the one-template.  In this case, define two predecessor templates for $\tau$, denoted $\tau_{p(0)}$ and $\tau_{p(1)}$, as follows:

\begin{itemize}
\item $\tau_{p(0)}$ is $\tau$ with the top level removed and $u_k(2^k-1)$ tacked on at the bottom.
\item $\tau_{p(1)}$ is $\tau$ with the top level removed and $\overline u_k(2^k-1)$ tacked on at the bottom. 
\end{itemize}

\noindent Also in this case, define two successor templates for $\tau$, denoted $\tau_{s(0)}$ and $\tau_{s(1)}$, as follows: 

\begin{itemize}
\item $\tau_{s(0)}$ is $\tau$ with the bottom level removed and $u_k(0)$ tacked on at the top.
\item $\tau_{s(1)}$ is $\tau$ with the bottom level removed and $\overline u_k(0)$ tacked on at the top.
\end{itemize}


If $\tau \in \mathcal B_k(X)$ is a basic template whose global cut is the top level (there are exactly two such basic templates), then define two predecessor templates for $\tau$, denoted $\tau_{p(0)}$ and $\tau_{p(1)}$, as follows:

\begin{itemize}
\item $\tau_{p(0)}$ is $\tau$ with the top level (the global cut) removed and $u_k(0)$ tacked on at the bottom.
\item $\tau_{p(1)}$ is $\tau$ with the top level (the global cut) removed and $\overline u_k(1)$ tacked on at the bottom.
\end{itemize}

\noindent Also in this case, define a single successor template for $\tau$, denoted $\tau_s$, to be the basic template that is $\tau$ with its bottom level removed, and with the level that would naturally follow the top level tacked on at the top.  For example, if $\tau \in \mathcal B_k(X)$ is \[\tau: u_k(1) \prec u_k(2) \prec \cdots \prec u_k(2^k-1) \prec \overline u_k(0),\] then $\tau_s$ is \[\tau_s: u_k(2) \prec u_k(3) \prec \cdots \prec u_k(2^k-1) \prec \overline u_k(0) \prec \overline u_k(1).\]

Because certain templates have multiple predecessor/successor templates, while others have only one, the following definition will be useful.

\begin{definition} \rm
If $\tau \in \mathcal B_k(X)$ is a basic template, then a {\em predecessor template for $\tau$} is any basic template of the form $\tau_p$, $\tau_{p(0)}$, or $\tau_{p(1)}$, as defined above.  A {\em successor template for $\tau$} is any basic template of the form $\tau_s$, $\tau_{s(0)}$, or $\tau_{s(1)}$, as defined above.
\end{definition}

\begin{lemma}
If $\tau \in \mathcal B_k(X)$ is a basic template, then all predecessor templates for $\tau$ agree in every level except possibly the bottom.  Also, all successor templates for $\tau$ agree in every level except possibly the top.

\end{lemma}

\subsection{Diminished Templates}  

Given a basic template $\tau \in \mathcal B_k(X)$ that is neither the zero-template nor the one-template, define two additional, {\em diminished} templates $\tau^-(d)$ and $\tau^-(u)$ as follows:

\begin{itemize}
\item $\tau^-(d)$ is $\tau$ with the global cut removed, and with an extra level tacked on at the bottom (the level that would naturally precede the bottom level).

\item $\tau^-(u)$ is $\tau$ with the global cut removed, and with an extra level tacked on at the top (the level that would naturally follow the top level).

\end{itemize}

\noindent For example, suppose $\tau \in \mathcal B_k(X)$ is given by  \[\tau: u_k(i) \prec u_k(i+1) \prec \cdots \prec u_k(2^k-1) \prec \overline u_k(0) \prec \overline u_k(1) \prec \cdots \prec \overline u_k(i-1).\]

\noindent Then the global cut is $\overline u_k(0)$, and the bottom level is $u_k(i)$.  The level that would naturally precede this bottom level is $u_k(i-1)$.  Therefore the diminished template $\tau^-(d)$ is given by  \[\tau^-(d): u_k(i-1) \prec u_k(i) \prec \cdots \prec u_k(2^k-1)  \prec \overline u_k(1) \prec \cdots \prec \overline u_k(i-1)\]
  
  Now suppose $\tau \in \mathcal B_k(X)$ is either the zero-template or the one-template.  Then there are two levels  that could naturally precede the bottom level (either $u_k(2^k-1)$ or $\overline u_k(2^k-1)$), as well as two levels that could naturally follow the top level (either $u_k(0)$ or $\overline u_k(0)$).  For this reason, we define four diminished templates, as follows:
  
  \begin{itemize}
  \item $\tau^-(d, 0)$ is $\tau$ with the global cut (which, in this case, is also the bottom level) removed, and with the level $u_k(2^k-1)$ tacked on in its place.
  \item $\tau^-(d, 1)$ is $\tau$ with the global cut (which, in this case, is also the bottom level) removed, and with the level $\overline u_k(2^k-1)$ tacked on in its place.
  \item $\tau^-(u, 0)$ is $\tau$ with the global cut (which, in this case, is also the bottom level) removed, and with the level $u_k(0)$ tacked on at the top.
    \item $\tau^-(u, 1)$ is $\tau$ with the global cut (which, in this case, is also the bottom level) removed, and with the level $\overline u_k(0)$ tacked on at the top.
  
  \end{itemize}

\noindent For example, if $\tau$ is the zero template, then $\tau^-(d, 1)$ is \[\tau^-(d, 1) = \overline u_k(2^k-1) \prec u_k(1) \prec u_k(2) \prec \cdots \prec u_k(2^k-1).\]

\noindent Let $\mathcal D_k(X)$ be the set of all diminished templates for the Morse system at stage $k$.  Note that all diminished templates in $\mathcal D_k(X)$ have height $2^k$.

\subsection{Augmented Templates} Given a basic template $\tau \in \mathcal B_k(X)$, define four additional, {\em augmented} templates as follows:

\begin{itemize}

\item $\tau^+(d, 0)$ is $\tau$ with an extra copy of $u_k(0)$ inserted directly before the global cut, and with the bottom level deleted.

\item $\tau^+(d, 1)$ is $\tau$ with an extra copy of $\overline u_k(0)$ inserted directly before the global cut, and with the bottom level deleted.

\item $\tau^+(u, 0)$ is $\tau$ with an extra copy of $u_k(0)$ inserted directly before the global cut, and with the top level deleted.

\item $\tau^+(u, 1)$ is $\tau$ with an extra copy of $\overline u_k(0)$ inserted directly before the global cut, and with the top level deleted.

\end{itemize}

\noindent For example, if $\tau \in \mathcal B_k(X)$ is given by  \[\tau: u_k(i) \prec u_k(i+1) \prec \cdots \prec u_k(2^k-1) \prec \overline u_k(0) \prec \overline u_k(1) \prec \cdots \prec \overline u_k(i-1),\] 

\noindent then $\tau^+(d, 0)$ is given by \[\tau^+(d, 0): u_k(i+1) \prec u_k(i+2) \prec \cdots \prec u_k(2^k-1) \prec u_k(0) \prec \overline u_k(0) \prec \cdots \prec \overline u_k(i-1).\] 

\noindent Let $\mathcal A_k(X)$ be the set of all augmented templates for the Morse system at stage $k$.  Note that all augmented templates in $\mathcal A_k(X)$ have height $2^k$.

\subsection{Missing Templates}  Given a basic template $\tau \in \mathcal B_k(X)$, define one additional, {\em missing} template, denoted $\tau^m$, to be $\tau$ with its bottom level removed.  For example, if $\tau$ is the zero-template, then \[\tau^m = u_k(1) \prec u_k(2) \prec \cdots \prec u_k(2^k-1).\]  Let $\mathcal M_k(X)$ be the set of all missing templates for the Morse system at stage $k$.  Note that all missing templates have height $2^k-1$.

\subsection{Extra Templates}  Given a basic template $\tau \in \mathcal B_k(X)$ that is neither the zero-template nor the one-template, define one additional, {\em extra} template, denoted $\tau^e$, to be $\tau$ with one extra level tacked on at the top (the level that would naturally follow the top level).  For example, if $\tau \in \mathcal B_k(X)$ is given by  \[\tau: u_k(i)\prec \cdots \prec u_k(2^k-1) \prec \overline u_k(0) \prec \overline u_k(1) \prec \cdots \prec \overline u_k(i-1),\] then $\tau^e$ is given by  \[\tau^e: u_k(i)  \prec \cdots \prec u_k(2^k-1) \prec \overline u_k(0) \prec \overline u_k(1) \prec \cdots \prec \overline u_k(i-1) \prec \overline u_k(i).\]  If $\tau$ is the zero-template or the one-template, then define two extra templates, as follows:

\begin{itemize}
\item $\tau^{e(0)}$ is $\tau$ with the level $u_k(0)$ tacked on at the top.
\item $\tau^{e(1)}$ is $\tau$ with the level $\overline u_k(0)$ tacked on at the top.
\end{itemize}

\noindent Let $\mathcal E_k(X)$ be the set of all extra templates for the Morse system at stage $k$.  Note that all extra templates in $\mathcal E_k(X)$ have height $2^k+1$.

\subsection{The Template Sets $\mathcal T_k$ and $\widetilde{\mathcal T}_k$} \label{tkdefinition}

Define \[\mathcal T_k = \mathcal B_k(X) \cup \mathcal D_k(X) \cup \mathcal A_k(X) \text{ and } \widetilde{\mathcal T}_k = \mathcal M_k(X) \cup \mathcal E_k(X).\]

\section{The Sequences $(\varepsilon_n)$ and $(k_n)$} \label{sequences}

Let $(\varepsilon_n)$ be a summable sequence.  Recall that we defined $k_0 = k_1 = 0$ and $k_2 = 2$.  Now define $k_n$, $n > 2$, by the following two-step recursion.   Given $k_n$, for $n$ even, define $k_{n+1}$ and $k_{n+2}$ as follows.

First choose $m_1 \in \n$ large enough so that \begin{equation} \label{m1choice}\frac{2(1+2^{k_n})}{2(1+2^{k_n})+m_1} < \varepsilon_{n},\end{equation} and pick $k_{n+1} \in \n$ large enough so that  \begin{equation} \label{knplusonechoice} 2^{k_{n+1}} \geq m_1\cdot 2^{k_{n}} + 2 \cdot (2^{k_n}+2^{k_{n}} \cdot 2^{k_{n}}). \end{equation} 

Then choose $m_2 \in \n$ large enough so that \begin{equation} \label{m2choice} \frac{2(1+2^{k_{n+1}})}{2(1+2^{k_{n+1}})+m_2} < \varepsilon_{n+1}, \end{equation} and pick $k_{n+2} \in \n$ large enough so that \begin{equation} \label{knplustwochoice} 2^{k_{n+2}} \geq m_2\cdot 2^{k_{n+1}} + 2 \cdot (2^{k_{n+1}}+2^{k_{n+1}} \cdot 2^{k_{n+1}}).\end{equation}

\noindent Note that inequalities (\ref{m2choice}) and (\ref{knplustwochoice}) are the same as (\ref{m1choice}) and (\ref{knplusonechoice}), only with $k_n$, $k_{n+1}$, $m_1$, and $\varepsilon_n$ replaced with $k_{n+1}$, $k_{n+2}$, $m_2$, and $\varepsilon_{n+1}$.

\section{From Stage $n$ to $n+2$: Frequently Used Notation} \label{freqnotation}

Assume the diagram in Figure~\ref{thefigure} has been built down to stage $n$, where $n$ is even.  Using our choices of $k_{n+1}$ and $k_{n+2}$ from Section \ref{sequences}, we then build the diagram down to stage $n+2$.  The construction is largely the same whether $n \equiv 2$ or $n \equiv 0 \mod 4$; Sections \ref{freqnotation} - \ref{blockpartitionsoftauhat} apply in either case.  

Define $J,J'\subset\mathbb Z$ by

  \[J = [0, 1, \ldots, 2^{k_{n+2}}-1] \text{ and } J' = [0, 1, \ldots, 2^{k_{n+1}}-1].\]    Define the {\em bottom and top global safe zones in $J$} to be the subintervals \[ [0, \ldots, 2^{k_{n+1}}+2^{k_{n+1}}\cdot 2^{k_{n+1}}-1] \text{ and } [2^{k_{n+2}}-(2^{k_{n+1}}+2^{k_{n+1}} \cdot 2^{k_{n+1}}), \ldots, 2^{k_{n+2}}-1] \] of $J$, respectively.  By (\ref{knplustwochoice}) the global safe zones are well-defined, and by (\ref{m2choice}) the fraction of $J$ in the global safe zones is less than $\varepsilon_{n+1}$.

Similarly, define the {\em bottom and top intermediate safe zones in $J'$} to be the subintervals \[ [0, \ldots, 2^{k_{n}}+2^{k_{n}}\cdot 2^{k_{n}}-1] \text{ and } [2^{k_{n+1}}-(2^{k_{n}}+2^{k_{n}} \cdot 2^{k_{n}}), \ldots, 2^{k_{n+1}}-1] \] of $J'$, respectively.  By \eqref{knplusonechoice} the intermediate safe zones are well-defined, and by  \eqref{m1choice} the fraction of $J'$ in the intermediate safe zones is less than $\varepsilon_{n}$.

 Let \[\mathcal B_{k_{n+2}} = \begin{cases} \mathcal B_{k_{n+2}}(X) & \mbox{if } n \equiv 2 \mod 4 \\ \mathcal B_{k_{n+2}}(Y) & \mbox{if } n \equiv 0 \mod 4 \end{cases}\] be the set of basic templates at stage $n+2$.  Let $\ds L = 2^{k_{n+2}}/2^{k_{n+1}}$ and $L' = 2^{k_{n+1}}/2^{k_n}$.  Fix a basic template $\tau \in \mathcal B_{k_{n+2}}$, represent $\tau$ with the interval $J$, and let $g = g(\tau)$ be the position in $J$ where the global cut occurs.  Let $a \in \{0, 1, \ldots, 2^{k_{n+1}}-1\}$ and $b \in \{0, 1, \ldots, 2^{k_n}-1\}$ be such that $g \equiv a \mod 2^{k_{n+1}}$ and $g \equiv b \mod 2^{k_n}$.   Let $\ds c = \left \lfloor \frac{a}{2^{k_n}} \right \rfloor$.

\subsection{Intermediate and Local Block Partitions of Basic Templates at Stage $n+2$} \label{blockpartitions}

In this section we define two partitions of $\tau$, which we call the intermediate and local block partitions of $\tau$.  We also define the intermediate and local cuts in $\tau$.

\begin{definition}[Intermediate Blocks in $J$] \rm

Depending on $a$, the {\em intermediate block partition of $J$} is either (\ref{simpleintblockpartition}) or (\ref{complicatedintblockpartition}) below:

If $a = 0$, then \begin{equation} \label{simpleintblockpartition} J = J(0) \cup J(1) \cup \cdots \cup J(L-1),\end{equation} where $J(m) = [m \cdot 2^{k_{n+1}}, \ldots, (m+1) \cdot 2^{k_{n+1}} - 1]$ for $0 \leq m \leq L-1$.

If $a \neq 0$, then \begin{equation} \label{complicatedintblockpartition}J = J(0) \cup J(1) \cup \cdots \cup J(L),\end{equation} where: \begin{itemize} \item $\ds J(0) = [0, \ldots, a-1]$ \item $\ds J(m) = [a + (m-1) \cdot 2^{k_{n+1}}, \ldots, a + m\cdot 2^{k_{n+1}}-1]$ for $1 \leq m \leq L - 1$ \item $\ds J(L) = [a + (L-1)\cdot 2^{k_{n+1}}, \ldots , 2^{k_{n+2}}-1]$.  \end{itemize}

Whether the intermediate block partition of $J$ takes the form (\ref{simpleintblockpartition}) or (\ref{complicatedintblockpartition}), the sub-intervals $J(m)\subset J$ are called the {\em intermediate blocks in $J$}.   

\end{definition}

\begin{definition}[Intermediate Blocks in $\tau$] \rm
The {\em intermediate block partition of $\tau$} is either \begin{equation} \label{simpleinttau} \tau = \tau(0) \cup \tau(1) \cup \cdots \cup \tau(L-1)  \end{equation} or \begin{equation} \label{complicatedinttau} \tau = \tau(0) \cup \tau(1) \cup \cdots \cup \tau(L), \end{equation} depending on whether the intermediate block partition of $J$ takes the form (\ref{simpleintblockpartition}) or (\ref{complicatedintblockpartition}), respectively.  Either way, the sub-templates $\tau(m) \subset \tau$ consist of those levels in $\tau$ that occur in positions from $J(m)$, and are called the {\em intermediate blocks in $\tau$}.  

\end{definition}

\begin{definition}[Intermediate Cuts in $\tau$] \rm

The first level in an intermediate block $\tau(m)$ is called an {\em intermediate cut in $\tau$}.

\end{definition}

\begin{definition}[Local Blocks in $J'$] \label{J'localblocks} \rm
The {\em local block partition of $J'$} is \begin{equation} \label{J'partition} J' = J'(0) \cup J'(1) \cup \cdots \cup J'(L' - 1),\end{equation} where $J'(i) = [i \cdot 2^{k_{n}}, \ldots, (i+1) \cdot 2^{k_{n}}-1]$ for $0 \leq i \leq L'-1$.  The sub-intervals $J'(m) \subset J'$, each of which has length $2^{k_n}$,  are called the {\em local blocks in $J'$}.  
\end{definition}

\begin{definition}[Local Blocks within Intermediate Blocks of height $2^{k_{n+1}}$] \label{localblocks1} \rm

If $\tau(m) \subset \tau$ is an intermediate block of height $2^{k_{n+1}}$, then the {\em local block partition of $\tau(m)$} is \begin{equation} \label{talltaupartition} \tau(m) = \tau(m, 0) \cup \tau(m, 1) \cup \cdots \cup \tau(m, L'-1),  \end{equation} where the sub-templates $\tau(m, i)$ consist of those levels in $\tau(m)$ that occur in positions from $J'(i)$, and are called the {\em local blocks in $\tau(m)$}.
\end{definition}

\begin{definition} \rm
Given an interval $I = [i, i+1, \ldots, i+\ell-1] \subset \z$ of length $\ell$ and $d \in \{0, \ldots, \ell-1\}$, denote the subinterval of $I$ consisting of the last $d$ integers in $I$ by $[I]^d$ and the subinterval of $I$ consisting of the first $\ell-d$ integers by $[I]_d$.  Namely, $[I]_d=I\setminus[I]^d$.


\end{definition}

%

\begin{definition}[Local Blocks in $\mbox{$[J]_a$}$ and $\mbox{$[J]^a$}$] \label{J'alocalblocks} \rm
The {\em local block partitions of $[J']_a$ and $[J']^a$} are \begin{equation} \label{J'_apartition} [J']_a = J'(0) \cup J'(1) \cup \cdots \cup J'(L'-c-2) \cup \left[ J'(L'-c-1)  \right]_b \end{equation} and \begin{equation} \label{J'^apartition} [J']^a = \begin{cases} \left[J'(L'-c-1) \right]^b \cup J'(L'-c) \cup \cdots \cup J'(L'-1) & \mbox{ if }a \neq 0 \\ \emptyset & \mbox{ if } a = 0.\end{cases} \end{equation} The sub-intervals in the partitions (\ref{J'_apartition}) and (\ref{J'^apartition}) are called the {\em local blocks in $[J']_a$ and $[J']^a$}.

\end{definition}

\begin{remark} \rm

If $a = 0$, then $b = c = 0$, so the local block partition (\ref{J'_apartition}) of $[J']_0$  is identical to the local block partition  (\ref{J'partition}) of $J'$.  Also $J' = [J']_0$.  Therefore Definition \ref{J'localblocks} is just a special case of Definition \ref{J'alocalblocks}.  

\end{remark}

\begin{definition}[Local Blocks within Intermediate Blocks of height $< 2^{k_{n+1}}$] \label{localblocks2} \rm

Intermediate blocks of height $< 2^{k_{n+1}}$ can only exist if $a \neq 0$; in this case, $\tau(L)$ and $\tau(0)$ are the only two such.  Represent $\tau(L)$ with $[J']_a$ and $\tau(0)$ with $[J']^a$.  Then the {\em local block partitions of $\tau(L)$ and $\tau(0)$} are \begin{equation} \label{tauLpartition} \tau(L) = \tau(L, 0) \cup \tau(L, 1) \cup \cdots \cup \tau(L, L'-c-2) \cup [\tau(L, L'-c-1)]_b  \end{equation} and \begin{equation} \label{tau0partition} \tau(0) = [\tau(0, L'-c-1)]^b \cup \tau(0, L'-c) \cup \cdots \cup \tau(0, L'-1), \end{equation} where the sub-templates in the partitions (\ref{tauLpartition}) and (\ref{tau0partition}) consist of those levels in $\tau(L)$ and $\tau(0)$ that occur in positions from the corresponding local blocks in $[J']_a$ and $[J']^a$, and are called the {\em local blocks in $\tau(L)$ and $\tau(0)$}.

\end{definition}

\begin{definition}[Local Blocks and Local Cuts in $\tau$]  \rm A {\em local block in $\tau$} is any local block from Definitions \ref{localblocks1} or \ref{localblocks2}.  The {\em local block partition of $\tau$} is the partition of $\tau$ into its local blocks. The first level in a local block is called a {\em local cut in $\tau$}.

\end{definition}

\subsection{Block Partitions of Non-Basic Templates at Stage $n+2$} 

Recall that each basic template $\tau \in \mathcal B_{k_{n+2}}$ has four variation types: diminished, augmented, missing, and extra.  And within a given variation type, there may be multiple templates.  But any such template is constructed by applying one or both of the following operations to $\tau$:  \begin{enumerate} \item Remove one level from the bottom of a local block in $\tau$, \item Insert one new level at the top (or bottom) of a local block in $\tau$. \end{enumerate}   

Therefore, if $\tau'$ is a diminished, augmented, missing, or extra version of $\tau$, then the intermediate and local block partitions of $\tau$ determine intermediate and local block partitions of $\tau'$, as follows:  

\begin{enumerate}
\item Suppose a level, $d$, is removed from the bottom of a local block $B$.  Let $C$ be the intermediate block that contains $B$.  Then, in the intermediate and local block partitions of $\tau'$, replace $B$ and $C$ with $B \setminus \{d\}$ and $C \setminus \{d\}$.  Leave all other intermediate and local blocks alone.

\item Suppose a level, $d$, is inserted at the top (resp., bottom) of a local block $B$.  Let $C$ be the intermediate block that contains $B$.  Then, in the intermediate and local block partitions of $\tau'$, replace $B$ and $C$ with $B \uplus \{d\}$ and $C \uplus \{d\}$, where, in the new order, $d$ is the top (resp., bottom) level in $B \uplus \{d\}$.  Leave all other intermediate and local blocks alone.

\end{enumerate}

\section{The Reordering Maps $\hat p_1$ $\hat p_2$} 

Recall the definition of a reordering map in Section \ref{reorderingmaps}.  We will employ two types of reordering maps: {\em global} and {\em intermediate}.  Roughly speaking, the {\em global} reordering map takes a small number of levels in $\tau$ that occur in positions from the bottom global safe zone in $J$ and moves them, one by one, up to the top global safe zone.  This has the effect of sliding all levels in $\tau$ that do not occur in the global safe zones (those in the ``middle part") down.  Here is the formal definition:

\begin{definition}[Global Reordering] \rm
The {\em global reordering map for $\tau$} is a map $\hat p_1= [J, J, J, J, p_1]$ where
\begin{enumerate}
\item If $a=0$, then $p_1$ and therefore $\hat p_1$ is the identity.
\item If $a\neq 0$ then $p_1$ is defined so that $\hat p_1$ takes the bottom levels in  $\tau(1), \tau(2), \ldots, \tau(a)$ and inserts them, in order, directly after the top levels in $\tau(L-a), \tau(L-a+1), \ldots, \tau(L-1)$.  Observe that this means $\hat p_1$ shifts all levels in $\tau$ that do not occur in the global safe zones down by exactly $a$ positions.
\end{enumerate}
More formally when $a\neq 0$, for each intermediate block $J(m)$ in $J$, let $s_m$ and $\ell_m$ denote the smallest and largest integers in $J(m)$, respectively. Then 
$p_1 = q_a \circ q_{a-1} \circ \cdots \circ q_1$, where, for $1 \leq m \leq a$, \[q_{m}(j) = \begin{cases} \ell_{L-a+m-1} & \mbox{ if }j = s_m - m + 1 \\ j - 1 & \mbox{ if } s_m - m + 1 < j \leq \ell_{L-a+m-1} \\ j & \mbox{ if } j < s_m - m + 1 \mbox{ or } j > \ell_{L-a+m-1}.   \end{cases} \]

%
%
%
%
%

\end{definition}

\begin{proposition} \label{intcutspostp1}
The intermediate cuts in $\hat p_1(\tau)$ that do not occur in the global safe zones occur in positions in $J$ that are congruent to $0$ mod $2^{k_{n+1}}$.
\end{proposition}

This follows immediately from the definitions and implies that the intermediate cuts in $\hat p_1(\tau)$ that do not occur in the global safe zones ``line up" with intermediate cuts in the zero template $\tau^\star = \mathcal P_{k_{n+2}}(0)$.  This in turn implies that the intermediate blocks in $\hat p_1(\tau)$ that do not occur in the global safe zones also line up with intermediate blocks in the zero template.

\subsection{Block Partitions of $\hat p_1(\tau)$} \label{p1hatpartitions}

If $a = 0$, then the intermediate block partition of $\tau$ is given by (\ref{simpleinttau}) and $\hat p_1(\tau) = \tau$.  In this case we define the intermediate block partition of $\hat p_1(\tau)$ to be identical to the intermediate block partition of $\tau$.  Formally  \begin{equation} \label{simpleintblockpart} \hat p_1(\tau) = [\hat p_1(\tau)](0) \cup [\hat p_1(\tau)](1) \cup \cdots \cup [\hat p_1(\tau)](L-1), \end{equation} where each $[\hat p_1(\tau)](m) = \tau(m)$.  In this $a=0$ case we define the the local block partition of $\hat p_1(\tau)$ to be identical to the local block partition of $\tau$ (see (\ref{talltaupartition})).  Formally, for $0 \leq m \leq L-1$, \begin{equation} \label{tallinttaupartition} [\hat p_1(\tau)](m) = [\hat p_1(\tau)](m, 0) \cup [\hat p_1(\tau)](m, 1) \cup \cdots \cup [\hat p_1(\tau)](m, L'-1)  \end{equation} where each $[\hat p_1(\tau)](m, i) = \tau(m, i)$.

Now suppose $a \neq 0$, so that the intermediate block partition of $\tau$ is given by (\ref{complicatedinttau}).   Then $\hat p_1$ takes the bottom levels in  $\tau(1), \tau(2), \ldots, \tau(a)$ and inserts them, in order, directly after the top levels in $\tau(L-a), \tau(L-a+1), \ldots, \tau(L-1)$.  Then we define the intermediate block partition of $\hat p_1(\tau)$ by \begin{equation} \label{complicatedintblockpart} \hat p_1(\tau) = [\hat p_1(\tau)](0) \cup [\hat p_1(\tau)](1) \cup \cdots \cup [\hat p_1(\tau)](L)  \end{equation} where: \begin{itemize} \item For $m \not \in \{1, \ldots, a\} \cup \{L-a, \ldots, L-1\}$, $[\hat p_1(\tau)](m) = \tau(m)$, and  \item For $m \in \{1, \ldots, a\}$, $[\hat p_1(\tau)](m)$ is $\tau(m)$ with its bottom level removed, and $[\hat p_1(\tau)](L-a+m-1)$ is $\tau(L-a+m-1)$ with the bottom level of $\tau(m)$ inserted at the top.  \end{itemize}  

If $a \neq 0$ and $m \not \in \{1, \ldots, a\} \cup \{L-a, \ldots, L-1\}$, then define the local block partition of $[\hat p_1(\tau)](m)$ to be identical to the local block partition of $\tau$.  Formally, if $m \not \in \{0, 1, \ldots, a\} \cup \{L-a, \ldots, L\}$, then the local block partition of $[\hat p_1(\tau)](m)$ is given by (\ref{tallinttaupartition}).  The local block partitions of $[\hat p_1(\tau)](L)$ and $[\hat p_1(\tau)](0)$ are \begin{equation} \label{localpartitionofLth} [\hat p_1(\tau)](L) = [\hat p_1(\tau)](L, 0) \cup \cdots \cup [\hat p_1(\tau)](L, L'-c-2) \cup \left[ [\hat p_1(\tau)](L, L'-c-1)\right]_b  \end{equation} and \begin{equation} \label{localpartitionof0th} [\hat p_1(\tau)](0) = \left[ [\hat p_1(\tau)]\right(0, L'-c-1)]^b \cup [\hat p_1(\tau)](0, L'-c) \cup \cdots \cup [\hat p_1(\tau)](0, L'-1)  \end{equation} where $[\hat p_1(\tau)](L, i) = \tau(L, i)$ for each $i$.  

Finally, if $a \neq 0$ and $m \in \{1, \ldots, a\}$, then define the local block partitions of $[\hat p_1(\tau)](m)$ and $[\hat p_1(\tau)](L-a+m-1)$ by \begin{equation} [\hat p_1(\tau)](m) = [\hat p_1(\tau)](m, 0) \cup [\hat p_1(\tau)](m, 1) \cup \cdots \cup [\hat p_1(\tau)](m, L'-1) \end{equation} and \begin{equation} [\hat p_1(\tau)](L-a+m-1) = [\hat p_1(\tau)](L-a+m-1, 0) \cup \cdots \cup [\hat p_1(\tau)](L-a+m-1, L'-1)   \end{equation} where: \begin{itemize} \item For $i \neq 0$, $[\hat p_1(\tau)](m, i) = \tau(m, i)$, \item For $i \neq L'-1$, $[\hat p_1(\tau)](L-a+m-1, i) = \tau(L-a+m-1, i)$, and \item $[\hat p_1(\tau)](m, 0)$ is $\tau(m, 0)$ with the bottom level removed, and $[\hat p_1(\tau)](L-a+m-1, L'-1)$ is $\tau(L-a+m-1, L'-1)$ with the bottom level of $\tau(m, 0)$ inserted at the top. \end{itemize}

\begin{proposition} \label{twocases}

Let $[\hat p_1(\tau)](m)$ be an intermediate block in $\hat p_1(\tau)$ that does not occur in a global safe zone.  Then the levels in $[\hat p_1(\tau)](m)$ occur in the same positions as those in the intermediate block $\tau^\star(\ell)$ in the zero template where

\[\ell = \begin{cases} m & \mbox{ if } a=0 \\ m-1 & \mbox{ if } a \neq 0. \end{cases}\]  
Therefore if $\tau$ is an odometer template then  $\zeta([\hat p_1(\tau)](m)) = \zeta(\tau^\star(\ell))$. If $\tau$ is a Morse template then
 either
$\zeta([\hat p_1(\tau)](m)) = \zeta(\tau^\star(\ell))$ or $\zeta([\hat p_1(\tau)](m)) = \overline{\zeta(\tau^\star(\ell))}$.


\end{proposition}

\begin{proof}
Follows from Proposition \ref{intcutspostp1}.

\end{proof}

\subsection{The Intermediate Reordering Map $\hat p_2$}

In this section we define a reordering map $\hat r_m$ for each intermediate block $[\hat p_1(\tau)](m)$ in $\hat p_1(\tau)$.  We then define the {\em intermediate reordering map} to be the concatenation, denoted $\hat p_2$, of the maps $\hat r_m$ (either $\hat p_2 = \hat r_0 \ast \cdots \ast \hat r_{L-1}$ or $\hat p_2 = \hat r_0 \ast \cdots \ast \hat r_L$).  Here is the formal definition.

\begin{definition}[The Intermediate Reordering Map] \label{intermediatereorderingmap} \rm

Depending on whether $a = 0$ or $a \neq 0$, define the {\em intermediate reordering map for $\tau$} to be the concatenation $\hat p_2 = \hat r_0 \ast \cdots \ast \hat r_{L-1}$ or $\hat p_2 = \hat r_0 \ast \cdots \ast \hat r_{L}$, respectively, where each $\hat r_m = [J', J', J', J', r_m]$ is defined as follows.  If $[\hat p_1(\tau)](m)$ is an intermediate block in $\hat p_1(\tau)$ that occurs in a global safe zone, then $r_m = $ identity.  If $[\hat p_1(\tau)](m)$ is an intermediate block in $\hat p_1(\tau)$ that does not occur in a global safe zone then, letting \[x = \begin{cases} m & \mbox{ if } a=0 \\ m-1 & \mbox{ if } a \neq 0, \end{cases}\] Proposition \ref{twocases} determines two cases: 

\begin{enumerate}
\item If $\zeta([\hat p_1(\tau)](m)) = \zeta(\tau^\star(x))$, then $r_m $ is the identity.

\item If $\zeta([\hat p_1(\tau)](m)) = \overline{\zeta(\tau^\star(x))}$, then $r_m$ is defined so that $\hat r_m$ takes the bottom levels in $[\hat p_1(\tau)](m, 1), \ldots, [\hat p_1(\tau)](m, 2^{k_n})$ and inserts them, in order, directly after the top levels in $[\hat p_1(\tau)](m, L' - 2^{k_n}), \ldots, [\hat p_1(\tau)](m, L'-1)$.  This shifts all other local blocks down by exactly $2^{k_n}$ positions.  More formally, for $1 \leq t \leq L'$, let $s_t$ and $\ell_t$ denote the smallest and largest integers in $J'(t)$, respectively.  Then let $r_m = q_{2^{k_n}} \circ q_{2^{k_n}-1} \circ \cdots \circ q_1$ where, for $1 \leq t \leq 2^{k_n}$, \[q_{t}(j) = \begin{cases} \ell_{L'-2^{k_n}+t-1} & \mbox{ if }j = s_t - t + 1 \\ j - 1 & \mbox{ if } s_t - t + 1 < j \leq \ell_{L-2^{k_n}+t-2} \\ j & \mbox{ if } j < s_t - t + 1 \mbox{ or } j > \ell_{L-2^{k_n}+t-2}.   \end{cases} \]

\end{enumerate}

\end{definition}




\subsection{Block Partitions of $\hat \tau = \hat p_2\circ \hat p_1(\tau)$} \label{blockpartitionsoftauhat}

Let $\hat \tau = \hat p_2 \circ \hat p_1(\tau)$ and, depending on whether $a = 0$ or $a \neq 0$, define the intermediate block partition of $\hat \tau$ to be either \[ \hat \tau = \hat \tau(0) \cup \hat \tau(1) \cup \cdots \cup \hat \tau(L-1) \;\; \mbox{ or } \;\; \hat \tau = \hat \tau(0) \cup \hat \tau(1) \cup \cdots \cup \hat \tau(L) \] respectively, where, for each $m$, $\ds \hat \tau(m) = \hat r_m\left( \left[ \hat p_1(\tau) \right](m) \right)$.  

Given an intermediate block $\hat \tau(m)$ in $\hat \tau$, if $r_m = $ identity, then define the local block partition of $\hat \tau(m)$ to be identical to the local block partition of $[\hat p_1(\tau)](m)$.  Denote the local blocks in $\hat \tau(m)$ by $\hat \tau(m, i)$, $[\hat \tau(m, i)]_b$, or $[\hat \tau(m, i)]^b$ depending on the form that $\hat \tau(m) = [\hat p_1(\tau)](m)$ takes (the various possible forms are described in Section \ref{p1hatpartitions}). 

If $r_m \neq $ identity, then $\ds r_m = q_{2^{k_n}}  \circ \cdots \circ q_1$, as in Definition \ref{intermediatereorderingmap}.  In this case, define the local block partition of $\hat \tau(m)$ to be \[\hat \tau(m) = \hat \tau(m, 0) \cup \hat \tau(m, 1) \cup \cdots \cup \hat \tau(m, L'-1)\] where: 

\begin{itemize}
\item $\hat \tau(m, 0)$ consists of the $2^{k_n}$ consecutive levels in $\hat \tau(m)$ that occur in positions from $[0, \ldots, 2^{k_n}-1]$ in $J'$;
\item For $1 \leq i \leq 2^{k_n}$, $\hat \tau(m, i)$ consists of the $2^{k_n}-1$ consecutive levels in $\hat \tau(m)$ that occur in positions from $[i \cdot 2^{k_n}-i+1, \ldots, (i+1)\cdot 2^{k_n} - i - 1]$ in $J'$;
\item For $2^{k_n}+1 \leq i \leq L' - 2^{k_n} - 2$, $\hat \tau(m, i)$ consists of the $2^{k_n}$ consecutive levels in $\hat \tau(m)$ that occur in positions from $[(i - 1) \cdot 2^{k_n}, \ldots, i \cdot 2^{k_n} - 1]$ in $J'$;
\item For $L' - 2^{k_n} - 1 \leq i \leq L' - 2$, $\hat \tau(m, i)$ consists of the $2^{k_n}+1$ consecutive levels in $\hat \tau(m)$ that occur in positions from \[[(i-1)\cdot 2^{k_n} + i - (L' - 2^{k_n}-1), \ldots, i \cdot 2^{k_n} + i - (L' - 2^{k_n}-1)]\] in $J'$; and
\item $\hat \tau(m, L'-1)$ consists of the $2^{k_n}$ consecutive levels in $\hat \tau(m)$ that occur in positions from $[(L'-1)2^{k_n}, \ldots, 2^{k_{n+1}}-1]$ in $J'$.

\end{itemize}

\section{The Good Set} \label{thegoodset}
All the machinery we have defined up to this point will be used to construct the partial interval bijections from Figure~\ref{thefigure}.  The global and intermediate reordering maps in particular are defined to guarantee the existence of \lq\lq good sets".  Recall from the introduction that these are subsets of $X$ and $Y$ defined for each stage of the construction on which all partial interval bijections match.  

Define the {\em good set in $J'$} to be the subset $G' \subset J'$ given by  \[G' = \bigcup_{t = 0}^{L'/2 - 2^{k_n} - 1}J'(2^{k_n} + 2t). \] Then $G'$ consists of every other local block in $J'$ that does not occur in the intermediate safe zones.  For $0 \leq s \leq L-1$, define $\ds G(s) = G' + s \cdot 2^{k_{n+1}},$  and \[G_{n+2} = \bigcup_{s = 2^{k_{n+1}}+1}^{L - (2^{k_{n+1}}2^{k_{n+1}} + 2^{k_{n+1}}) - 1} G(s).\] Note that $G_{n+2} \subset J$; we call $G_{n+2}$ the {\em good set within $J$ at stage $n+2$}.  

For $n$ congruent to $0\mod 4$ ($2\mod 4$) we set $\mathcal G_{n+2} \subset \mathcal Q_{k_{n+2}}$ ($\mathcal H_{n+2}\subset\mathcal P_{k_{n+2}}$) to consist of those levels in $\mathcal Q_{k_{n+2}}$ ($\mathcal P_{k_{n+2}}$) that occur in positions from $G_{n+2}\subset J$.  We call $\mathcal G_{n+2}$ or $\mathcal H_{n+2}$ the {\em good set at stage $n+2$}.     Note that, because the global and intermediate safe zones make up a very small proportion of $J$, $G_{n+2}$ consists of roughly half of $J$.  Therefore $\nu(\mathcal G_{n+2})$ and $\mu(\mathcal H_{n+2})$ are both approximately $\frac12$.


\begin{proposition} \label{everyotherproperty}
Let $\tau_1$ and $\tau_2$ be basic templates in $\mathcal B_{n+2}(X)$ $\left(\mathcal B_{n+2}(Y)\right)$, and let $g \in \mathcal G_{n+2}$ $(g \in \mathcal H_{n+2})$.  Let $c_1 \in \hat \tau_1$ and $c_2 \in \hat \tau_2$ be the levels that occur in position $g$ within $\hat \tau_1$ and $\hat \tau_2$.  Then $\pi \circ \zeta(c_1) = \pi \circ \zeta(c_2)$.  
\end{proposition}

\begin{proof}
The global reordering map was defined so that, after global reordering, the intermediate block structures of all templates line up outside of the global safe zone (see Proposition \ref{twocases}).  For the odometer system, then, after global reordering, every local block that is not in the global safe zones matches the corresponding local block in the zero template, in the sense that it has the same image under $\pi \circ \zeta$.

For the Morse system, the intermediate reordering map was defined so that, after both global and intermediate reordering, every other local block that is neither in the global nor intermediate safe zones moved down by exactly one complete local block.  As a consequence of the combinatoric structure of the Morse system, now every other such local block matches the corresponding local block in the zero template.  To see why this is the case, consider the Morse sequence and its flip: \begin{eqnarray*} \mbox{Morse sequence} & = & .\fbox{0}1\fbox{1}0\fbox{1}0\fbox{0}1\fbox{1}0\fbox{0}1\fbox{0}1\fbox{1}0 \cdots \\ \mbox{Flip} & = & .1\fbox{0}0\fbox{1}0\fbox{1}1\fbox{0}0\fbox{1}1\fbox{0}1\fbox{0}0\fbox{1} \cdots  \end{eqnarray*}  Notice that, if we shift the flip to the left by one coordinate, then it matches the Morse sequence in every other coordinate.  The same would be true for sequences of substitution blocks (just replace $0$'s and $1$'s with blocks $\theta^n(0)$ and $\theta^n(1)$).  
\end{proof}

\begin{proposition} \label{independenceofgoodsets}
Let $n < m$ be natural numbers congruent to $0 \mod 4$ $(2\mod 4)$.  Then $\mathcal G_{n}$ and $\mathcal G_{m}$ $(\mathcal H_n$ and $\mathcal H_m)$ are independent events in $Y$ $(X)$.
\end{proposition}

\begin{proof}We only give the proof in the case where $n$ and $m$ are congruent to $0\mod 4$.  The other case follows similarly.  
Since $n < m$, we have $2^{k_m} = p \cdot 2^{k_n}$, where $p = 2^{k_m - k_n}$.  Then $\mathcal Q_{k_m}$ is partitioned into $p$ ``copies" of $\mathcal Q_{k_n}$:  \begin{equation} \label{qmpartition} \mathcal Q_{k_m} = \mathcal L_1 \cup \mathcal L_2 \cup \cdots \cup \mathcal L_p, \end{equation} where each $\mathcal L_i$ is a union of exactly $2^{k_n}$ consecutive levels in $\mathcal Q_{k_m}$, and $\pi(\mathcal L_i) =  \mathcal Q_{k_n}$ for each $i$.  Because $\mathcal G_m$ is a union of local blocks from stage $m - 2$ and $n < m - 2$, we can write $\mathcal G_m$ as a disjoint union \begin{equation} \label{gmdisjoint} \mathcal G_m = \bigcup_{i \in I} \mathcal L_i\end{equation} for some $I \subset \{1, \ldots, p\}$.  Because (\ref{qmpartition}) is obtained by cutting and stacking $\mathcal Q_{k_n}$ $p$ times, we have $\nu(\mathcal G | \mathcal L_i) = \nu(\mathcal G_n)$, and therefore $\nu(\mathcal G_n \cap \mathcal L_i) = \nu(\mathcal L_i) \nu(\mathcal G_n)$, for each $i$.  Since (\ref{gmdisjoint}) is a disjoint union,

\begin{equation*} \mu(\mathcal G_n | \mathcal G_m)  =  \frac{\nu(\mathcal G_n \cap \mathcal G_m)}{\nu(\mathcal G_m)}  =  \frac{1}{\nu(\mathcal G_m)} \cdot \sum_{i \in I} \nu(\mathcal G_n \cap \mathcal L_i) = \frac{1}{\nu(\mathcal G_m)} \cdot \sum_{i \in I} \nu(\mathcal L_i) \nu(\mathcal G_n) = \nu(\mathcal G_n). \end{equation*} 
Therefore  $\mathcal G_m$ and  $\mathcal G_n$
are independent.

\end{proof}

\section{Stage 4 of the Induction} 

\subsection{Heads and Tails}\label{headsandtails}
Suppose $\tau \in \mathcal B_{k_4}(X)$ and $b \neq 0$.   Then the local block partitions of $\hat \tau(0)$ and $\hat \tau(L)$ are \[ \hat \tau(0) = [\hat \tau(0, L'-c-1)]^b \cup \hat \tau(0, L'-c) \cup \cdots \hat \tau(0, L'-1)\]         and \[\hat \tau(L) = \hat \tau(L, 0) \cup \cdots \hat \tau(L, L'-c-2) \cup [\hat \tau(L, L'-c-1)]_b,\] as described Section \ref{blockpartitionsoftauhat}, and in (\ref{localpartitionof0th}) and (\ref{localpartitionofLth}).              Notice that the very first local block in $\hat \tau$ has height $b \neq 2^{k_2} $, i.e., it is a {\em partial} block.  Correspondingly, the bottom $b$ levels in $\mathcal Q_{k_4}$ can be thought of as the top $b$ levels in a collection of templates $\omega \in \Omega_{k_2}$.  Define \[\tau_{tail} = \{\omega \in \Omega_{k_2} \; : \; \mbox{the top $b$ levels in $\omega$ are the bottom $b$ levels in $\mathcal Q_{k_4}$}\}.\]  Similarly, define \[\tau_{head} = \{\omega \in \Omega_{k_2} \; : \; \mbox{the bottom $2^{k_2} - b$ levels in $\omega$ are top $2^{k_2} - b$ levels in $\mathcal Q_{k_4}$}\}.\]  If $b = 0$, then define $\ds \tau_{tail} = \tau_{head} = \emptyset$.

\subsection{Maps to Basic Templates}

Given $\tau \in \mathcal B_{k_4}(X)$, in this section we define a collection of partial interval bijections of the form \begin{equation} \label{basicPIBform} \phi_{\tau} = \phi_{\tau}^{tail} \ast \phi_{\tau}^{body} \ast \phi_{\tau}^{head}.\end{equation} There will be one such partial interval bijection for each pair $(\omega_1, \omega_2) \in \tau_{tail} \times \tau_{head}$.  We call $\phi_{\tau}^{tail}$ and $\phi_{\tau}^{head}$ the {\em bottom} and {\em top sticky notes} of $\phi_\tau$.  If $b = 0$, then each $\phi_\tau$ is defined on an interval of length $2^{k_4}$.  However, if $b \neq 0$, then each $\phi_\tau$ is defined on an interval of length approximately $2^{k_4} + 2^{k_2}$.  

\subsubsection{Sticky Notes $\phi_\tau^{tail}$ and $\phi_\tau^{head}$} \label{stage4stickynotes}

If $b = 0$, then define just one top sticky note and just one bottom sticky note, namely, the trivial partial interval bijection $[I, J, A, B, f]$ where $I = J = \emptyset$.         

If $b \neq 0$, then, for each $\omega \in \tau_{tail}$, define $\phi_\tau^{tail} = \phi_\omega^{-1}$.  To conserve notation, since the particular choice of $\omega$ will not matter for our construction, the expression $\phi_\tau^{tail}$ does not indicate dependence on $\omega$.                                                                                                                                                                                                                                                                                                                                         

Similarly, for each $\omega \in \tau_{head}$, define $\phi_\tau^{head} = \phi_\omega^{-1}$.

\subsubsection{The Body Map $\phi_\tau^{body}$} \label{stage4bodymap}

The local block partition of $\hat \tau$ determines a partition of $J = [0, 1, \ldots, 2^{k_4}-1]$ into subintervals, which in turn determines a partition of $\mathcal Q_{k_4}$, which we call the {\em local block partition of $\mathcal Q_{k_4}$}.  The local block partition of $\mathcal Q_{k_4}$ has the form \[ \begin{cases} \omega_1 \cup \omega_2 \cup \cdots \cup \omega_{LL'} & \mbox{if }b=0 \\ \omega_1 \cup \omega_2 \cup \cdots \cup  \omega_{LL'+1} & \mbox{if } b \neq 0,  \end{cases}\] where the $\omega_i$ are consecutively occurring sets of levels in $\mathcal Q_{k_4}$.  If $b = 0$, then, for each $i$, $\pi(\omega_i)$ is a template from $\Omega_{k_2} \cup \widetilde \Omega_{k_2}$.  In this case, define \[\widetilde{\widetilde{\phi_\tau^{body}}} = \phi_{\omega_1}^{-1} \ast \phi_{\omega_2}^{-1} \ast \cdots \ast \phi_{\omega_{LL'}}^{-1}.\] If $b \neq 0$, then for $2 \leq i \leq LL'$, $\pi(\omega_i)$ is a template from $\Omega_{k_2} \cup \widetilde \Omega_{k_2}$.  In this case, define \[\widetilde{\widetilde{\phi_\tau^{body}}} = \phi_{\omega_2}^{-1} \ast \phi_{\omega_3}^{-1} \ast \cdots \ast \phi_{\omega_{LL'}}^{-1}.\] 

\begin{lemma} \label{stage4matching}
Let $c \in \mathcal G_4$ be a level in $\mathcal Q_{k_4}$ from the good set at stage $4$.  Then, given $\tau_1$ and $\tau_2$ in $\mathcal B_{k_4}(X)$, either both $\ds \widetilde{\widetilde{\phi_{\tau_1}^{body}}}$ and $\ds \widetilde{\widetilde{\phi_{\tau_2}^{body}}}$ are undefined on $c$, or else $\ds \pi \circ \zeta \circ \widetilde{\widetilde{\phi_{\tau_1}^{body}}}(c) = \pi \circ \zeta \circ \widetilde{\widetilde{\phi_{\tau_2}^{body}}}(c)$.
\end{lemma}

\begin{proof}
By construction, together with Proposition \ref{everyotherproperty}.
\end{proof}

It follows from Lemma \ref{stage4matching} that it is possible to extend the domain of definition of $\ds \widetilde{\widetilde{\phi_\tau^{body}}}$ to a partial interval bijection, call it $\ds \widetilde{\phi_\tau^{body}}$, between the set of all levels in $\mathcal G_4$ and those levels in $\hat \tau$ that occur in positions from $G_4$.  Moreover, the same bijection can be used for all reordered templates $\hat \tau$ so that, given $\tau_1$ and $\tau_2$ in $\mathcal B_{k_4}(X)$ and $c \in\mathcal G_4$, $\ds \pi \circ \zeta \circ \widetilde{\phi_{\tau_1}^{body}}(c) = \pi \circ \zeta \circ \widetilde{\phi_{\tau_2}^{body}}(c)$.

Finally, define \[\phi_\tau^{body} = \begin{cases} (p_2 \circ p_1)^{-1} \circ \widetilde{\phi_\tau^{body}} & \mbox{ if }b = 0 \\ (\overline{p_2 \circ p_1})^{-1} \circ \widetilde{\phi_\tau^{body}} & \mbox{ if } b \neq 0, \end{cases}\] where $\overline{p_2 \circ p_1}$ denotes the restriction of $p_2 \circ p_1$ to the subinterval of $\overline J \subset J$ corresponding to $\omega_2 \cup \cdots \cup \omega_{LL'}$. (Note that $\overline{p_2 \circ p_1}$ is well-defined because $p_2 \circ p_1$ is the identity outside of $\overline J$.) 

\begin{proposition} \label{stage4nobottomortop}
Neither the bottom level nor the top level of $\mathcal Q_{k_4}$ is in the domain of any $\phi_\tau$.
\end{proposition}

\begin{proof}
Given $\phi_\tau$ of the form $\ds \phi_\tau = \phi_\tau^{tail} \ast \phi_\tau^{body} \ast \phi_\tau^{head}$ determined by $(\omega_1, \omega_2) \in \tau_{tail} \times \tau_{head}$, the bottom level in $\mathcal Q_{k_4}$ occurs as the cut in $\omega_1$.  But by proposition \ref{nobottomtoporcut}, this level is not in the range of the stage 2 map $\phi_{\omega_1}$.  Therefore it is not in the domain of $\phi_{\tau}^{tail}$.  
\end{proof}

\begin{proposition} \label{stage4nocut}
Given $\tau \in \mathcal B_{k_4}(X)$, the global cut in $\tau$ is not in the range of any $\phi_\tau$.
\end{proposition}

\begin{proof}
Follows from Proposition \ref{nobottomortop}.
\end{proof}

\begin{proposition}  \label{stage4matchingprop}
Let $c \in \mathcal G_4$.  Then, given $\tau_1$ and $\tau_2$ in $\mathcal B_{k_4}(X)$ and partial interval bijections $\ds \phi_{\tau_1}$ and $\ds \phi_{\tau_2}$ of the form (\ref{basicPIBform}), both $\ds \phi_{\tau_1}$ and $\ds \phi_{\tau_2}$ are defined on $c$, and $\ds \pi \circ \zeta \circ \phi_{\tau_1}(c) = \pi \circ \zeta \circ \phi_{\tau_2}(c)$.

\end{proposition}

\begin{proof}
Follows from Lemma \ref{stage4matching} together with the definition of the partial interval bijections $\phi_\tau$.

\end{proof}

\subsection{Partial Interval Bijections}  We now define maps to diminished, augmented, missing, and extra templates, and establish analogues of Lemmas \ref{stage2diminishedmatching} and \ref{stage2augmentedmatching}, which are needed to establish analogues of Propositions \ref{possiblebodyoverlaps} and \ref{bodyoverlapsvalid} in stage 8.

\subsubsection{Maps to Diminished Templates}
Given $\tau \in \mathcal B_{k_4}(X)$ that is neither the zero-template nor the one-template, define $\phi_{\tau^-(u)}$ to match with $\phi_\tau$, and define $\phi_{\tau^-(d)}$ to match with $\phi_{\tau_p}$.  Here we mean that the top and bottom sticky notes as well as the body maps all match.  Similar to stage 2, the fact that these definitions are possible follows from Proposition \ref{stage4nocut}.  

If $\tau \in \mathcal B_{k_4}(X)$ is either the zero-template or the one-template, then for $i \in \{0, 1\}$, define $\phi_{\tau^-(u, i)}$ to match with $\phi_\tau$, and define $\phi_{\tau^-(d, i)}$ to match with $\phi_{\tau_{p(i)}}$.

\begin{lemma} \label{stage4dimlemma}
If $\tau \in \mathcal B_{k_4}(X)$ is neither the zero-template nor the one-template, then $\phi_{\tau^-(d)}$ matches with $\phi_{\tau_p^-(u)}$.  If $\tau \in \mathcal B_{k_4}(X)$ is either the zero-template or the one-template, then for $i \in \{0, 1\}$, $\phi_{\tau^-(d, i)}$ matches with $\phi_{\tau_p^-(i)(u)}$.
\end{lemma}

\begin{proof}
By construction.
\end{proof}

\subsubsection{Maps to Augmented Templates} 

Given $\tau \in \mathcal B_{k_4}(X)$ and $i \in \{0, 1\}$, define $\phi_{\tau^+(d, i)}$ to match with $\phi_\tau$.  If $\tau$ is neither the zero-template nor the one-template, then define $\phi_{\tau^+(u, i)}$ to match with $\phi_{\tau_p}$.  If $\tau$ is the zero-template or the one-template, then define $\phi_{\tau^+(u, i)}$ to match with $\phi_{\tau_{p(i)}}$.

\begin{lemma} \label{stage4auglemma}
Given $i \in \{0, 1\}$, and $\tau \in \mathcal B_{k_4}(X)$, if $\tau$ is neither the zero-template nor the one-template, then $\phi_{\tau^+(u, i)}$ matches with $\phi_{\tau_p^+(d, i)}$.  If $\tau$ is the zero-template or the one-template, then $\phi_{\tau^+(u, i)}$ matches with $\phi_{\tau^+_{p(i)}(d, i)}$.
\end{lemma}

\begin{proof}
By construction.
\end{proof}

\subsubsection{Maps to Missing and Extra Templates} \label{stage4missingextra}

Given $\tau^m \in \mathcal M_{k_4}(Y)$, define $\tau^m_{head} = \tau_{head}$.  Note that, if $\omega \in \tau^m_{head}$, then $\omega^-(u) \in \tau^m_{head}$.  Define \[\tau^m_{tail} = \{\omega \in \Omega_{k_2} \; : \; \mbox{the top $b-1$ levels in $\omega$ are the bottom $b-1$ levels in $\mathcal Q_{k_4}^m$}\}.\] Note that $\tau_{tail} \subset \tau^m_{tail}$, but $\tau_{tail} \neq \tau^m_{tail}$ since, in particular, if $\omega \in \tau_{tail}$, then $\omega^-(d) \in \tau^m_{tail}$, but $\omega^-(d) \not \in \tau_{tail}$.

Define \[\phi_{\tau^m} = \phi_{\tau^m}^{tail} \ast \phi_{\tau^m}^{body} \ast \phi_{\tau^m}^{head}\] as follows.  For each $\omega \in \tau_{head}^m$, define $\phi_{\tau^m}^{head} = \phi_\omega^{-1}$, and, for each $\omega \in \tau_{tail}^m$, define $\phi_{\tau^m}^{tail} = \phi_\omega^{-1}$.  Define $\phi_{\tau^m}^{body} = \phi_\tau^{body}$.

Similarly, given $\tau \in \mathcal B_{k_4}(X)$ that is neither the zero-template nor the one-template, define $\tau_{tail}^e = \tau_{tail}$.  Note that, if $\omega \in \tau_{tail}^e$, then $\omega^+(d) \in \tau_{tail}^e$.  Define \[\tau_{head}^e = \{\omega \in \Omega_{k_2} \; : \; \mbox{the bottom $2^{k_2}-b+1$ levels in $\omega$ are the top $2^{k_2}-b+1$ levels in $\mathcal Q_{k_4}^e$}\}.\] Note that, if $\omega \in \tau_{head}$, then $\omega^+(u) \in \tau_{head}^e$.

\section{Stage 6 of the induction}

\subsection{Reordering Maps, Block Partitions, the Good Set, and Heads and Tails}\label{stage6bigpicture}
Stage 6 is analogous to stage 4, but with one new layer of complexity:  Whereas the maps in stage 4 were defined as {\em concatenations} of maps from stage 2, the maps in stage 6 will be {\em overlapping concatenations} of the maps from stage 4.  The top and bottom sticky notes defined in stage 4 are used to glue these overlapping concatenations together.  We need to verify that there are enough sticky notes defined in stage 4 to choose from so that these overlapping concatenations are well-defined.

The global reordering map in stage 6, denoted $\hat p_1$, is the same as it was in stage 4 except, of course, now $J = [0, 1, \ldots, 2^{k_6}-1]$.  In stage 4, an intermediate reordering map was also used to ensure that every other local block that does not occur in a safe zone has the same image under $\pi \circ \zeta$ as the corresponding local block in the zero-template (see Proposition \ref{everyotherproperty}).  But an intermediate reordering map is not needed in stage 6 because the odometer has just one canonical tower at each stage---not two.  So in stage 6, let the intermediate reordering map $\hat p_2$ be simply the identity.  Given $\omega \in \mathcal B_{k_6}(Y)$, let $\hat \omega = \hat p_2 \circ \hat p_1(\omega)$.

Define the intermediate and local block partitions of $\hat p_1(\omega)$, as well as the intermediate block partition of $\hat \omega$, in exactly the same way that they were defined in stage 4 (see Sections \ref{p1hatpartitions} and \ref{blockpartitionsoftauhat}) except, of course, that $\tau$ is replaced with $\omega$.  Given an intermediate block $\hat \omega(m)$ in $\hat \omega$, because the intermediate reordering map $\hat p_2$ is simply the identity, just like in stage 4, define the local block partition of $\hat \omega(m)$ to be identical to the local block partition of $[\hat p_1(\omega)](m)$.

The good set within $J$ at stage $6$, $G_6$, is defined in Section \ref{thegoodset}.  Let $\mathcal G_6$, called the {\em good set at stage $6$}, consist of those levels in $\mathcal P_{k_6}$ that occur in positions from $G_6$.  Note that Proposition \ref{everyotherproperty} holds in stage 6.

Given $\omega \in \mathcal B_{k_6}(Y)$, the sets $\omega_{tail}$ and $\omega_{head}$ are defined analogously to the definitions in Section \ref{headsandtails} (just replace $\tau$, $\omega$, $k_2$, and $k_4$ with $\omega$, $\tau$, $k_4$, and $k_6$, respectively).  However, there is an added layer of complexity in stage 6:  If $\tau \in \omega_{head}$ (or $\tau \in \omega_{tail}$), then $\tau$ has head and tail sets of its own, $\tau_{head}$ and $\tau_{tail}$.

\subsection{Maps to Basic Templates}

Much like in stage 6, given $\omega \in \mathcal B_{k_6}(Y)$, the partial interval bijections $\phi_\omega$ take the form \begin{equation} \label{stage6basicPIBform} \phi_\omega = \phi_\omega^{tail} \tilde \ast \phi_\omega^{body} \tilde \ast \phi_\omega^{head},\end{equation} where $\ds \tilde \ast$ represents {\em overlapping concatenation} as defined in Definition \ref{overlapconcatenation1}.  There is one such partial interval bijection for each pair $(\tau_1, \tau_2) \in \omega_{tail} \times \omega_{head}$.  The maps $\phi_\omega^{tail}$, $\phi_\omega^{body}$, and $\phi_\omega^{head}$ are defined analogously to the definitions in Sections \ref{stage4stickynotes} and \ref{stage4bodymap} except that within the body map $\phi_\omega^{body}$ (Section \ref{stage4bodymap}), concatenations ($\ast$) are replaced with overlapping concatenations ($\tilde \ast$).  We now show that the overlapping concatenations within $\phi_\omega^{body}$ can be glued together.  Suppose $b \neq 0$ (the case that $b = 0$ is nearly identical).  Then, following what was done in Section \ref{stage4bodymap}, we define $ \ds \widetilde{\widetilde{\phi_\omega^{body}}}$ to take the form \begin{equation} \label{phiomegabodydecomp} \widetilde{\widetilde{\phi_\omega^{body}}} = \phi_{\tau_2}^{-1} \tilde \ast \phi_{\tau_3}^{-1} \tilde \ast \cdots \tilde \ast \phi_{\tau_{LL'}}^{-1}.\end{equation}  

\begin{proposition} \label{possiblebodyoverlaps}
With the templates $\tau_i$ from (\ref{phiomegabodydecomp}) listed in order of overlapping concatenation ($\tau_2 \prec \tau_3 \prec \cdots \prec \tau_{LL'}$), the following six types of successive pairs can occur:
\begin{enumerate}
\item $\tau \prec \tau'$ where $\tau, \tau' \in \{\mathcal P_{k_4}(0), \mathcal P_{k_4}(1)\}$,
\item $\tau \prec \tau$ where $\tau \in \mathcal B_{k_4}(X) \setminus \{\mathcal P_{k_4}(0), \mathcal P_{k_4}(1)\}$,
\item $\tau_{s(i)}^m \prec \tau$, where $i \in \{0, 1\}$ and $\tau = \mathcal P_{k_4}(i)$,  
\item $\tau_s^m \prec \tau^m$, where $\tau \in \mathcal B_{k_4}(X) \setminus \{\mathcal P_{k_4}(0), \mathcal P_{k_4}(1)\}$,
\item $\tau \prec \tau_{s(j)}^e$, where $i, j \in \{0, 1\}$ and $\tau = \mathcal P_{k_4}(i)$, and
\item $\tau^e \prec \tau_s^e$, where $\tau \in \mathcal B_{k_4}(X) \setminus \{\mathcal P_{k_4}(0), \mathcal P_{k_4}(1)\}$.
\end{enumerate}
\end{proposition}

\begin{proof}
By construction.
\end{proof}

\begin{proposition} \label{bodyoverlapsvalid}
In any of the six cases of Proposition \ref{possiblebodyoverlaps}, top and bottom sticky notes can be chosen so that the overlapping concatenation of the corresponding partial interval bijections $\phi_{\tau}^{-1}$ are well-defined.
\end{proposition}

\begin{proof}
Case 1 is trivial because partial interval bijections to the zero- and one-templates do not have sticky notes (there is no overlap to worry about).  For case 2, observe that if $\omega \in \mathcal B_{k_2}(Y)$ is a basic template from stage 2 such that $\omega \in \tau_{tail}$, then $\omega \in \tau_{head}$ as well, so we can glue $\phi_\tau^{-1}$ together with itself by picking such $\phi_\omega^{-1}$ on the overlap.  For case 3, observe that there exists $\omega \in \mathcal B_{k_2}(Y)$ such that $\ds \phi_{\omega^-(d)}^{-1} \in \phi_\tau^{tail}$ and $\ds \phi_{\omega_p^-(u)}^{-1} \in \phi_{\tau_{s(i)}^m}^{head}$.  Lemma \ref{stage2diminishedmatching} then guarantees that there is a bottom sticky note on $\phi_\tau$ that matches with a top sticky note on $\phi_{\tau_s^m}$.  Cases 4-6 are similar.
\end{proof}

The following shows that \eqref{stage6basicPIBform} is well defined.

\begin{proposition}
Top and bottom sticky notes can be chosen so that the overlapping concatenations $\phi_\omega^{tail} \tilde \ast \phi_\omega^{body}$ and $\phi_\omega^{body} \tilde \ast \phi_\omega^{head}$ are well-defined.
\end{proposition}

\begin{proof}
Similar to the proof of Proposition \ref{bodyoverlapsvalid}.
\end{proof}

The following propositions are analogous to propositions from stage 4.

\begin{proposition}
Neither the bottom level nor the top level of $\mathcal P_{k_6}(0)$ or $\mathcal P_{k_6}(1)$ is in the domain of any $\phi_\omega$.
\end{proposition}

\begin{proof}
Follows from Proposition \ref{stage4nocut}.
\end{proof}

\begin{proposition}
Given $\omega \in \mathcal B_{k_6}(Y)$, the global cut in $\omega$ is not in the range of any $\phi_\omega$.
\end{proposition}

\begin{proof}
Follows from Proposition \ref{stage4nobottomortop}.
\end{proof}

\begin{proposition}
Let $g \in \mathcal G_6$ and let $c \in \mathcal P_{k_6}(i)$ be a level that occurs in position $g$, where $i \in \{0, 1\}$.  Then, given $\omega_1$ and $\omega_2$ in $\mathcal B_{k_6}(Y)$ and partial interval bijections $\phi_{\omega_1}$ and $\phi_{\omega_2}$ of the form (\ref{phiomegabodydecomp}), both $\phi_{\omega_1}$ and $\phi_{\omega_2}$ are defined on $c$, and $\pi \circ \zeta \circ \phi_{\omega_1}(c) = \pi \circ \zeta \circ \phi_{\omega_2}(c)$.
\end{proposition}

\begin{proof}
Proposition \ref{everyotherproperty} can be used to show that Lemma \ref{stage4matching} holds in stage 6 (with appropriate notational modifications).  The proposition follows.
\end{proof}

\subsection{Maps to Diminished, Augmented, Missing, and Extra Templates} \label{stage6missingextra}

The definitions of these modified maps are analogous to those in stage 4.  Maps to missing and extra templates are needed to define the maps at stage 8 in places where individual levels have been deleted or inserted.  Analogues of Lemmas \ref{stage4dimlemma} and \ref{stage4auglemma} hold in stage 6, and are used to glue sticky notes together in stage 10.

\section{Completing the induction} \label{stage8}

The essential components of the induction have now been established.  Stages $8, 12, 16, \ldots$ are analogous to stage $4$, while stages $10, 14, 18, \ldots$ are analogous to stage $6$.  In each stage $n \geq 8$, concatenations of the (inverses of) the partial interval bijections at stage $n-2$ are glued together using Definition \ref{overlapconcatenation2}.

\section{Convergence of Partial Interval Bijections} \label{convergenceofPIBs}

Given a level $c \in \mathcal G_4$, Proposition \ref{stage4matchingprop} guarantees that $\pi \circ \zeta \circ \phi_\tau(c)$ is a level in $\mathcal P_{k_2}$ that does not depend on which $\tau \in \mathcal B_{k_4}$ is used.  Moreover, Proposition \ref{stage4matchingprop} holds in every stage $n \equiv 0 \mod 4$, which permits us to define $\phi^n: \mathcal Q_{k_n} \rightarrow \mathcal P_{k_{n-2}}$ to be the restriction of $\pi \circ \zeta \circ \phi_\tau$ to $\mathcal G_n$.  Similarly, for $n \equiv 2 \mod 4$, we can define $\phi^n: \mathcal P_{k_n} \rightarrow \mathcal Q_{k_{n-2}}$ to be the restriction of $\pi \circ \zeta \circ \phi_\omega$ to $\mathcal G_n$.

Given $n \equiv 0 \mod 4$ and $y \in Y$, let $d_{k_n}(y)$ denote the unique $k_n$-canonical cylinder in $Y$ that contains $y$.  Let $\mathcal H \subset Y$ be the set of $y \in Y$ such that $d_{k_n}(y) \in \mathcal G_{n}$ for infinitely many $n$.

Similarly, given $n \equiv 2 \mod 4$ and $x \in X$, let $c_{k_n}(x)$ denote the unique $k_n$-canonical cylinder in $X$ that contains $x$.  Let $\mathcal G \subset X$ be the set of $x \in X$ such that $c_{k_n}(x) \in \mathcal G_{n}$ for infinitely many $n$.

\begin{theorem}
The sets $\mathcal G$ and $\mathcal H$ are $G_\delta$ sets of full measure.
\end{theorem}

\begin{proof}
By Proposition \ref{independenceofgoodsets}, the sets $\mathcal G_n$ for $n \equiv 0 \mod 4$ are independent with respect to $\nu$.  Therefor $\nu(\mathcal G) = 1$ by the Borel-Cantelli Lemma.  Moreover, the sets $\mathcal G_n$ are open, so $\mathcal G$ is a $G_\delta$ subset of $Y$.  The argument for $\mathcal H$ is similar.
\end{proof}

\begin{lemma} \label{containmentlemma}
Given $t > 0$, if $n \equiv 0 \mod 4$ and $d \in \mathcal G_n$ and $d' \in \mathcal G_{n + 4t}$ are levels such that $d' \subset d$, then $\phi^{n+4t}(d') \subset \phi^n(d)$. Also, if $n \equiv 2 \mod 4$, $c \in \mathcal G_n$, $c' \in \mathcal G_{n+4t}$, and $c' \subset c$, then $\phi^{n+4t}(c') \subset \phi^n(c)$.
\end{lemma}

\begin{proof}
Given $\tau \in \mathcal T_{k_{n+4t}} \cup \widetilde{\mathcal T}_{k_{n+4t}}$, the map $\phi_\tau$ at stage $n + 4t$ is an extension of a concatenation of the maps $\phi_\tau$ at stage $n$.  It follows that \[\pi \circ \phi^{n+4t}(d') = \phi^n \circ \pi(d') = \phi^n(d),\] where, depending on the context, $\pi$ refers either to the map $\pi: \mathcal P_{k_{n+4t-2}} \rightarrow \mathcal P_{k_{n-2}}$ or to the map $\pi: \mathcal Q_{k_{n+4t}} \rightarrow \mathcal Q_{k_n}$.  The argument when $n \equiv 2 \mod 4$ is nearly identical.
\end{proof}

Given $y \in \mathcal H$, let $(n(i))_{i \in \n}$ be the increasing sequence of indices, each congruent to $0 \mod 4$, such that each $d_{k_{n(i)}}(y) \in \mathcal G_{k_{n(i)}}$.  Then, by Lemma \ref{containmentlemma}, the levels $\ds (\phi^{n(i)}(d_{k_{n(i)}}(y)))_{i \in \n}$ form a nested sequence.  It follows that there is a unique point in the intersection $\ds \bigcap_{i \in \n} \phi^{n(i)}(d_{k_{n(i)}}(y))$.  Similarly, if $(m(i))_{i \in \n}$ is the analogous sequence of indices for $x \in \mathcal G$, then there is a unique point in the intersection $\ds \bigcap_{i \in \n} \phi^{m(i)}(c_{k_{n(i)}}(x))$.  This permits the following definition.

\begin{definition} \rm
Given $x \in \mathcal G$, let $\phi(x)$ be the unique point in the intersection $\ds \bigcap_{i \in \n} \phi^{m(i)}(c_{k_{n(i)}}(x))$.  Given $y \in \mathcal H$, let $\psi(y)$ be the unique point in the intersection $\ds \bigcap_{i \in \n} \phi^{n(i)}(d_{k_{n(i)}}(y))$.
\end{definition}

\begin{theorem}
The maps $\phi$ and $\psi$ are continuous in the relative topologies on $\mathcal G$ and $\mathcal H$.
\end{theorem}

\begin{proof}
Let $x \in \mathcal G$ and $\varepsilon > 0$.  Choose $n \equiv 2 \mod 4$ large enough so that $\ds 1/2^{k_{n-2}} < \varepsilon$ and such that $x \in \mathcal G_{k_n}$.  Choose $\delta > 0$ small enough so that $\rho(x, x') < \delta$ implies $c_{k_n}(x) = c_{k_n}(x')$.  In particular, $\rho(x, x') < \delta$ implies $x' \in \mathcal G_{k_n}$.  It then follows from Lemma \ref{containmentlemma} that, for all $t > 0$ such that $c_{k_{n+4t}}(x) \in \mathcal G_{k_{n+4t}}$, \[\phi^{n + 4t}(c_{k_{n+4t}}(x)) \subset \phi^n(c_{k_n}(x))\] and, for all $\hat t > 0$ such that $c_{k_{n + 4\hat t}}(x') \in \mathcal G_{k_{n+4 \hat t}}$, \begin{eqnarray*} \phi^{n+4\hat t}(c_{k_{n+4 \hat t}}(x')) & \subset & \phi^n(c_{k_n}(x'))  = \phi^n(c_{k_n}(x)). \end{eqnarray*}  Therefore \[\rho(\phi(x), \phi(x')) < \frac{1}{2^{k_{n-2}}} < \varepsilon.\]  The continuity of $\psi$ is proved similarly.
\end{proof}

\begin{theorem}
The maps $\phi$ and $\psi$ are measure preserving.
\end{theorem}

\begin{proof}
Fix a natural number $m \equiv 0 \mod 4$ and a level $d$ in $\mathcal Q_{k_m}$.  We wish to show that $\mu(\phi^{-1}(d)) = \nu(d) = 1/2^{k_m}$.  Given $n > 0$, let  \[J_n(d) = \{ \mbox{levels } c \in \mathcal P_{m + 4n - 2} \; : \; \zeta \circ \phi_\omega(c) = d\; \; \forall \omega \in \Omega_{k_{m+4n-2}} \cup \widetilde \Omega_{k_{m+4n-2}}\},\] where $\zeta$ is the map $\zeta: \Omega_{k_{m+4n-2}} \cup \widetilde \Omega_{k_{m+4n-2}} \rightarrow \mathcal Q_{k_m}$.  

Let $D_n = \mathcal G_{m+4n-2}$ and $E_n = D_n \cap J_n(d)$.  Observe that \begin{equation} \label{Enequation} \frac{|E_n|}{|D_n|} = \frac{1}{2^{k_m}}.\end{equation}  As discussed in Section \ref{thegoodset}, for large $n$, $D_n$ consists of roughly half of the levels in $\mathcal P_{k_{m+4n-2}}$, so we can (conservatively) assume $D_n$ consists of at least $1/3$ of them.  Then \begin{equation} \label{Dnequation} |D_n| \geq \underbrace{2^{k_{m+4n-2}}/3 + 2^{k_{m+4n-2}}/3}_{1/3 \mbox{ of each tower}} = 2 \cdot 2^{k_{m+4n-2}}/3.\end{equation}  For $1 \leq a \leq n-1$, let $\widehat D_{n-a}$ denote the set of levels in $\mathcal P_{k_{m+4n-2}}$ that are contained in levels from $\mathcal G_{k_{m+4(n-a)-2}}$, and recursively define $\ds D_{n-a} = \widehat D_{n-a} \setminus \bigcup_{a'=0}^{a-1} D_{n-a'}$ and $\ds E_{n-a} = D_{n-a} \cap J_n(d)$.  Observe that \begin{equation} \label{Enminusaequation} \frac{|E_{n-a}|}{|D_{n-a}|} = \frac{1}{2^{k_m}}\end{equation} and \begin{equation} \label{Dnminusaequation} |D_{n-a}| \geq \frac{1}{3} \left(\underbrace{2 \cdot 2^{k_{m+4n-2}}}_{2 \mbox{ towers}} - \sum_{a'=0}^{a-1} | D_{n-a'}|  \right).\end{equation}  (Again, $1/3$ is a conservative lower bound---it is actually closer to $1/2$).  It follows from (\ref{Dnequation}) and (\ref{Dnminusaequation}) that \[ \sum_{a = 0}^{n-1} |D_{n-a}| \geq 2 \cdot 2^{k_{m+4n-2}} \left( 1 - \left( \frac{2}{3} \right)^n \right).\]  It now follows from (\ref{Enequation}) and (\ref{Enminusaequation}) that \[ |J_n(d)|  \geq  \sum_{a=0}^{n-1} |E_{n-a}|  =  \frac{1}{2^{k_m}} \sum_{a = 0}^{n-1} |D_{n-a}|  \geq  2 \cdot 2^{k_{m+4n-2}} \cdot \frac{1 - \left(\frac{2}{3} \right)^n}{2^{k_m}}.  \] This implies that \[\mu(J_n(d)) \geq \frac{1}{2^{k_m}} - \left( \frac{1}{2^{k_m}} \right) \cdot \left(\frac{2}{3} \right)^n.\]  Letting $n \rightarrow \infty$ gives $\mu(\phi^{-1}(d)) \geq \nu(d)$.  This being true for each level $d \in \mathcal Q_{k_m}$ trivially implies that $\mu(\phi^{-1}(d)) = \nu(d)$.

\end{proof}

\begin{theorem}
If $x \in \mathcal G$ and $\phi(x) \in \mathcal H$, then $\psi(\phi(x)) = x$.  Similarly, if $y \in \mathcal H$ and $\psi(y) \in \mathcal G$, then $\phi(\psi(y)) = y$.
\end{theorem}

\begin{proof}
Let $\varepsilon > 0$ and let $n \equiv 2 \mod 4$ be such that $\phi(x) \in \mathcal H_{k_{n+2}}$ and $\ds 1/2^{k_n} < \varepsilon$.  Let $t \geq 1$ be such that $x \in \mathcal G_{k_{n + 4t}}$.  Then \[ \phi^{n+2} \circ \pi \circ \phi^{n+4t}(c_{k_{n+4t}}(x))  =  \pi(c_{k_{n+4t}}(x))  =  c_{k_n}(x)\] where, depending on the context, $\pi$ refers either to the map $\ds \pi: \mathcal Q_{k_{n + 4t - 2}} \rightarrow \mathcal Q_{k_{n+2}}$ or to the map $\pi: \mathcal P_{k_{n+4t}} \rightarrow \mathcal P_{k_n}$.  Since $\phi(x) \in \phi^{n + 4t}(c_{k_{n+4t}}(x))$, we have $\phi(x) \in \pi \circ \phi^{n+4t}(c_{k_{n+4t}}(x))$, which implies $\psi(\phi(x)) \in c_{k_n}(x)$.  Therefore $\ds \rho(x, \psi(\phi(x))) < 1/2^{k_n} < \varepsilon$.  The second statement is proved similarly.

\end{proof}

\section{Kakutani Equivalence} \label{kakequivalence}

Let \[X_1 = \bigcap_{i \in \z} T^i(\mathcal G) \;\;\; \mbox{ and } \;\;\; Y_1 = \bigcap_{i \in \z} S^i(\mathcal H).\]  Then $X_1$ and $Y_1$ are invariant $G_\delta$ subsets of full measure.  Let $X_0 = X_1 \cap \phi^{-1}(Y_1)$ and $Y_0 = \phi(X_0)$.  Then $X_0$ and $Y_0$ are full measure subsets because $\phi$ is measure preserving.  And $X_0$ and $Y_0$ are $G_\delta$ subsets because $\phi$ is continuous in the relative topology on $X_1$.  In this section we show that $X_0$ and $Y_0$ are invariant and that $\phi: X_0 \rightarrow Y_0$ is an orbit equivalence that is a conjugacy when restricted to $\mathcal G_2$, the good set at stage $2$.

\begin{lemma} \label{orbitsintoorbits}
For $x \in X_1$, $\phi$ maps the $T$-orbit of $x$ into the $S$-orbit of $\phi(x)$.  And for $y \in Y_1$, $\psi$ maps the $S$-orbit of $y$ into the $T$-orbit of $\psi(y)$.
\end{lemma}

\begin{proof}
Let $x = x_1 \in X_1$ and $x_2 = T^{-r}(x_1)$ for some $r > 0$.  Recall that the bottom global safe zone at stage $n$ has height $h(n) := 2^{k_{n-1}} + 2^{k_{n-1}} \cdot 2^{k_{n-1}}$.  Pick $n_0 \equiv 2 \mod 4$ such that $h(n_0) > r$ and $x_1 \in \mathcal G_{k_{n_0}}$.  Then $x_1$ and $x_2$ are in the same tower in $\mathcal P_{k_{n_0}}$ and $c_{k_{n_0}}(x_2) = T^{-r}(c_{k_{n_0}}(x_1))$.  Moreover, since each $\mathcal G_{k_n}$ consists of complete local blocks from stage $n-2$, for each $n > n_0$ ($n \equiv 2 \mod 4$) such that $x_1 \in \mathcal G_{k_n}$, we have $x_2 \in \mathcal G_{k_n}$ and $c_{k_n}(x_2) = T^{-r}(c_{k_n}(x_1))$.  Fix such $n$.  For $i \in \{1, 2\}$, let $d_i \in \mathcal Q_{k_{n-2}}$ be such that $\phi^n(c_{k_n}(x_i)) = d_i$.  Let $t \in \z$ be such that $|t| < 2^{k_{n-2}}$ and $d_1 = S^t(d_2)$.

Let $m > n$ ($m \equiv 2 \mod 4$) such that $c_{k_m}(x_1) \in \mathcal G_{k_m}$.  For $i \in \{1, 2\}$, let $e_i \in \mathcal Q_{k_{m-2}}$ be levels such that $\phi^m(c_{k_m}(x_i)) = e^i$.  Then because the partial interval bijections $\phi_\omega$ at stage $m$ are extensions of concatenations of those at stage $n$, we have $e_1 = S^t(e_2)$.  Since $m$ was arbitrary, it follows that $\phi(x_1) = S^t(\phi(x_2))$.  Therefore $\phi$ maps the backward $T$-orbit of $x$ into the $S$-orbit of $\phi(x)$.  By similar argument, $\phi$ maps the forward $T$-orbit of $x$ into the $S$-orbit of $\phi(x)$.  The argument for $\psi$ is also similar.
\end{proof}

\begin{theorem} \label{X0theorem}
The sets $X_0 \subset X$ and $Y_0 \subset Y$ are invariant $G_\delta$ subsets of full measure, and $\phi: X_0 \rightarrow Y_0$ carries $T$-orbits bijectively to $S$-orbits.
\end{theorem}

\begin{proof}
We have already seen that $X_0$ and $Y_0$ are $G_\delta$ subsets of full measure.  Let $x \in X_0$.  Then $x \in X_1$ and $\phi(x) \in Y_1$ by definition.  Let $x' = T^r(x)$ for some $r \in \z$.  Then $x' \in X_1$ because $X_1$ is $T$-invariant.  And $\phi(x') \in Y_1$ by Lemma \ref{orbitsintoorbits} (and because $\phi(x) \in Y_1$).  Therefore $x' \in X_0$, so $X_0$ is $T$-invariant.  Similarly, $Y_0$ is $S$-invariant.

Now suppose $y$ is a point in the $S$-orbit of $\phi(x)$.  Then $y \in Y_1 \subset \mathcal H$, so $\psi(y)$ is in the orbit of $x$.  Hence $\psi(y) \in \mathcal G$.  Therefore $\phi$ carries the orbit of $x$ onto the orbit of $\phi(x)$.

\end{proof}

\begin{theorem}
The map $\phi$ is a conjugacy between the two induced maps $T_{\mathcal G_2}$ and $S_{\phi(\mathcal G_2)}$.  Both $\mathcal G_2$ and $\phi(\mathcal G_2)$ are nearly clopen.  Therefore $\phi$ is a nearly continuous Kakutani equivalence of $T$ and $S$.
\end{theorem}

\begin{proof}

Let $x_1$ and $x_2 = T^r(x_1)$ be two points in $\mathcal G_2 \cap X_0$.  Then, by Theorem \ref{X0theorem}, $\phi(x_2) = S^k(\phi(x_1))$ for some $k \in \z$.  We wish to show that $r$ and $k$ have the same sign.  This will imply that $\phi$ restricted to $\mathcal G_2 \cap X_0$ is order preserving on orbits, and hence a conjugacy between the induced maps.

As we saw in the proof of Lemma \ref{orbitsintoorbits}, $c_{k_n}(x_2) = T^r(c_{k_n}(x_1))$ for all sufficiently large $n \equiv 2 \mod 4$.  Let $\hat n$ be minimal among such $n$.  If $\hat n = 2$, then $r$ and $k$ automatically have the same sign because each partial interval bijection $\phi_\omega$ at stage $2$ maps the levels in $\mathcal G_2$ in an order-preserving way, and this is then carried through the diagram via concatenations.

If $\hat n > 2$, then because $\hat n$ is minimal, $c_{k_{\hat n - 4}}(x_2)$ and $c_{k_{\hat n-4}}(x_1)$ must lie in different towers in $\mathcal P_{k_{\hat n - 4}}$.  And the partial interval bijections at stage $\hat n$ are extensions of concatenations of those at stage $\hat n - 4$.  So if $r > 0$, then the stage-($\hat n - 4$) partial interval bijection that acts on $x_1$ in stage $\hat n$ comes before the stage-($\hat n - 4$) partial interval bijection that acts on $x_2$ in stage $\hat n$.  This order-preservation at stage $\hat n$ is then carried through the diagram via concatenations, so $k > 0$.  Similarly, $r < 0$ implies $k < 0$.  

\end{proof}

\end{document}